\documentclass[11pt]{amsart}
\usepackage{amssymb,amsmath,amsthm,amsfonts,mathrsfs,graphicx}
\usepackage{mathtools}
\usepackage{subcaption}

\usepackage[margin=3cm]{geometry}
\usepackage[colorlinks]{hyperref}
\usepackage{tikz}
\usepackage{xcolor}

\title[The damped Euler--Monge--Amp\`ere system]{On the Damped Euler--Monge--Amp\`ere System with Radial Symmetry:\\ Critical Thresholds and Large-Time Behavior}

\author{Kunhui Luan}
\address[Kunhui Luan]{\newline Department of Mathematics, \ 
 University of South Carolina, 1523 Greene St., Columbia, SC 29208, USA}
\email{kunhui.luan@sc.edu}

\thanks{\textit{Acknowledgment.} The author is grateful to Changhui Tan, Qiyu Wu, Siming He, Jesse Singh and Kyle Liss for numerous productive discussions and insightful feedback throughout the development of this work. The author also thanks two anonymous referees for their valuable comments that significantly improved the work. This work is partially supported by NSF grant DMS-2238219 and the George Johnson Fellowship.}

\makeatletter
\@namedef{subjclassname@2020}{\textup{2020} Mathematics Subject Classification}
\makeatother


\newtheorem{theorem}{Theorem}[section]
\newtheorem{lemma}[theorem]{Lemma}

\newtheorem{proposition}[theorem]{Proposition}

\theoremstyle{definition}

\theoremstyle{remark}
\newtheorem{remark}{Remark}[section]

\numberwithin{equation}{section}

\def\R{\mathbb{R}}

\def\dd{\mathrm{d}}

\def\eps{\varepsilon}

\def\CTsub{\widehat\Sigma_\flat} 
\def\CTsup{\widehat\Sigma_\sharp}
 
\def\L{\mathcal{L}}

\renewcommand{\geq}{\geqslant}

\renewcommand{\leq}{\leqslant}
\renewcommand{\le}{\leqslant}

\def\pa{\partial}

\def\F{\mathbf{F}}
\def\f{f}

\def\O{\mathbf{\Omega}}

\def\u{\mathbf{u}}
\def\v{\mathbf{v}}
\def\x{\mathbf{x}}

\def\grad{\nabla}

\def\div{\grad\cdot}
\def\eps{\varepsilon}

\def\tr{\text{trace}}

\def\Id{\mathbb{I}}

\def\d{\textnormal{d}}
\numberwithin{equation}{section}

\newcommand\numberthis{\addtocounter{equation}{1}\tag{\theequation}}

\newcommand{\CL}{\mathcal{L}}

\begin{document}
\allowdisplaybreaks

\begin{abstract}
We investigate the global well-posedness and large-time dynamics of the pressureless Euler--Monge--Amp\`ere (EMA) system with velocity damping in multidimensions, subject to radially symmetric initial data. We first establish the phenomenon of critical thresholds, where subcritical initial data maintain global regularity, and supercritical initial data lead to finite time singularity formation. We provide two methods for constructing these thresholds: a refined spectral dynamics approach based on \cite{liu2002spectral} and a comparison principle based on Lyapunov functions introduced in \cite{bhatnagar2020critical2}. 

A key finding is that the inclusion of linear damping, in certain regimes, effectively removes the initial density lower bound required in the undamped case \cite{tadmor2022critical}. More precisely, the subcritical region admits smooth non-vacuum initial data with arbitrarily small positive density. Furthermore, for subcritical initial data, we establish the large-time behavior of the system. Our results unify and extend existing theories for 1D Euler--Poisson system and undamped multidimensional EMA system with radial symmetry.

 \end{abstract}

\maketitle 

\tableofcontents

\section{Introduction}\label{sec:intro}
We study the following pressureless compressible Euler equation with two forces 
\begin{subequations}\label{eqs:EMA}
\begin{align}
 \pa_t\rho+\div(\rho\u)&=0,\label{eq:density}\\
 \pa_t(\rho\u)+\div(\rho\u\otimes\u)&=-\kappa\rho\grad\phi-\beta\rho\u,\label{eq:momentum}\\
 \det(\Id-D^2\phi)&=\rho\label{eq:MA},
\end{align}
\end{subequations}
subject to initial data
\begin{equation}\label{eq:initial_data}
  \rho(\x,t=0) = \rho_0(\x),\; \u(\x,t=0)=\u_0(\x),\;\phi(\x,t=0) = \phi_0(\x).
\end{equation}
Here $\rho(\cdot,t): {\mathbb R}^n \mapsto {\mathbb R}_+$,
 $\u(\cdot,t): {\mathbb R}^n \mapsto {\mathbb R}^n$ and
 $\phi(\cdot,t): {\mathbb R}^n \mapsto {\mathbb R}$ denote the density, velocity and potential, respectively.  Without loss of generality, we fix the potential assuming $\phi(0)=0$. 
 
 The system is known as the damped \emph{Euler--Monge--Amp\`ere} (EMA) system. The right-hand side of the momentum equation \eqref{eq:momentum} includes two types of interacting forces: the \emph{attractive-repulsive force} modeled by the Monge-Amp\`ere equation and a \emph{velocity damping}. Both constants $\kappa$ and $\beta$ are taken to be positive, with the former representing a repulsive force and the latter characterizing the damping strength.

 The EMA system was first introduced in a discrete version by Y. Brenier in \cite{brenier2000derivation}, with a subsequent kinetic counterpart, the Vlasov--Monge--Amp\`ere system, proposed and analyzed in \cite{brenier2004geometric}. The development of the EMA system is fundamentally linked to the geometric framework established by V. Arnold in \cite{arnold1966geometrie}, who characterized the incompressible Euler equations on a fixed domain $\Omega$ as the geodesics of the right-invariant metric on the volume-preserving diffeomorphism group $\mathcal{D}_\mu(\Omega)$, based on the Lagrangian representation of fluid flows. Building on this, G. Loeper \cite{loeper2005quasi} formally introduced the EMA system as a geometric approximation to the incompressible Euler equations, specifically by approximating the geodesics on $\mathcal{D}_\mu(\Omega)$.

\subsection{Comparison with the Euler--Poisson system.}
In addition to the EMA system, G. Loeper considered the following classical Euler--Poisson (EP) system in the same work \cite{loeper2005quasi}, 
\begin{subequations}\label{eqs:EP}
\begin{align}
 \pa_t\rho+\div(\rho\u)&=0,\label{eq:density_EP}\\
 \pa_t(\rho\u)+\div(\rho\u\otimes\u)&=-\kappa\rho\grad\phi,\label{eq:momentum_EP}\\
 -\Delta \phi&=\rho-1\label{eq:Poisson}.
\end{align}
\end{subequations}
It was shown that both systems are asymptotically close and converge to the incompressible Euler equations in the quasi-neutral limit. Specifically, the EMA system can be viewed as a nonlinear generalization of the EP system in the force equation. Indeed, we consider a perturbation near the equilibrium state $(\rho,\u) = (1,\mathbf{0})$, where the Monge--Amp\`ere equation linearizes to the Poisson equation,
\[
\rho = \det(I - \eps D^2\phi) = 1-\eps \Delta \phi + \mathcal{O}(\eps^2).
\]

The Euler--Poisson system models various physical phenomena, including astrophysics, semiconductor and plasma dynamics. The stability of the Euler-Poisson system near the equilibrium state was analyzed in
\cite{guo1998smooth,guo2011global,ionescu2013euler,guo2017absence,fajman2025critical}. Notably, in the one-dimensional case, the Monge--Amp\`ere
 equation \eqref{eq:MA} reduces to the Poisson equation $-\pa_{xx}\phi = \rho - 1$, rendering the EMA and EP systems equivalent. The global well-posedness theory for the 1D pressureless Euler-Poisson system with a constant background charge was first established in \cite{engelberg2001critical}. The global behavior of the solution exhibits a \emph{critical threshold} phenomenon---specifically, a dichotomy in the initial data: \emph{subcritical data} lead to global smooth solutions, while \emph{supercritical data} result in finite-time singularity formation (see Theorem~\ref{thm:thresholds} for precise definitions). Furthermore,
\cite{engelberg2001critical} provides an explicit and sharp expression for this threshold. Additional work on related models of the 1D EP system includes \cite{ bhatnagar2023critical, rozanova2020exact}.

A generalization of the Euler-Poisson system arises when the background charge is not constant. The critical threshold phenomenon for this system has been recently investigated in \cite{bhatnagar2020critical, choi2025critical, rozanova2025repulsive, luan2026euler}. In particular, \cite{choi2025critical} establishes subcritical and supercritical regions under the presence of additional velocity damping, which was further extended in \cite{luan2026euler} to a nonlocal velocity alignment term. A series of \emph{comparison principles} were developed based on Lyapunov functions introduced in \cite{bhatnagar2020critical2}, leading to the identification of non-trivial subcritical and supercritical regions, which underlies one of the main approaches in this paper.

 Several works have explored weak solutions to the Euler-Poisson system, particularly in 1D \cite{gangbo2009euler,natile2009wasserstein, brenier2013sticky, carrillo2023equivalence}. The system has also been studied with the inclusion of pressure \cite{guo1998smooth, tadmor2008global, guo2011global, fajman2025critical, liu2026large}. Global theory for smooth solutions is known for radially symmetric data \cite{wei2012critical,tan2021eulerian,carrillo2023existence,jang2012two,rozanova2023behavior,rozanova2024collision}. For general data, critical thresholds are established for a 2D restricted Euler-Poisson (REP) system in \cite{tadmor2003critical, liu2002spectral}, but the system in two or three dimensions are less understood, and global regularity for the Euler-Poisson system remains an open question.

The close relationship between the EMA and EP systems naturally raises the question of whether a critical threshold phenomenon exists for the EMA system. An affirmative answer was first provided by E. Tadmor and C. Tan in \cite{tadmor2022critical}, who presented a comprehensive analysis on the global regularity of EMA system with radial symmetry. They established a sharp and explicit critical threshold phenomenon. A salient feature of their result, however, is that global regularity in $n$ dimensions necessitates a lower bound on the initial density, specifically $\rho_0 > 2^{-n}$ (See \cite[Remark 3.5]{tadmor2022critical}).

Building on the work of Tadmor and Tan \cite{tadmor2022critical}, this paper investigates the global regularity and large-time dynamics of the EMA system under radial symmetry with the inclusion of velocity damping. Our primary contribution is two-fold: we first derive  \emph{sharp} and \emph{explicit} critical thresholds in Theorem~\ref{thm: GWP} and ~\ref{thm:GWP_lya}, and second, we establish an optimal uniform--in--time exponential decay rate of the solution toward the equilibrium state $(\rho,\u) = (1,\mathbf{0})$ in Theorem ~\ref{thm:decay}. We demonstrate that the introduction of damping effectively removes the initial density lower bound required in \cite{tadmor2022critical} in certain regimes.

We establish the threshold conditions through two complementary frameworks: first, by utilizing the spectral dynamics introduced in \cite{liu2002spectral} through a refined phase plane analysis; and second, via a comparison principle based on Lyapunov functions. Furthermore, leveraging the explicit phase plane analysis, we characterize the large-time asymptotic behavior of the system. Our results serve as a unifying framework: in the one-dimensional case, we recover the damped Euler--Poisson thresholds obtained in \cite{bhatnagar2020critical2}, while in the absence of damping, we retrieve the results of \cite{tadmor2022critical}.

\subsection{Outline of the paper.}
The remainder of this paper is organized as follows. In Section 2, we briefly review Eulerian dynamics with radial symmetry, following the general framework established in \cite{tan2021eulerian}, and present our main results for the damped EMA system. In Section 3, we derive the spectral dynamics governing the system. These dynamics are then utilized in Section 4 to construct the critical thresholds and analyze the large-time asymptotic behavior. In Section 5, we present an alternative construction of the threshold regions using Lyapunov functions. Finally, an energy estimate proof of the local well-posedness result is provided in Appendix \ref{sec:LWP}.

\section{Background and statement of main results.}

\subsection{Standing assumptions}
\label{sec:standing-assumptions}
In this paper, we consider no-swirl, radially symmetric solutions
of the form
\begin{equation}\label{eq:radial}
  \rho(\x,t)=\rho(r,t),\qquad
  \u(\x,t)=\frac{\x}{r}u(r,t),\qquad
  \phi(\x,t)=\phi(r,t),\qquad r=|\x|>0.
\end{equation}
We denote the radial variable by $r$, and $\rho$, $u$ and $\phi$ are scalar functions defined in $\R_+\times\R_+$. We work with smooth radial initial data of the form
\begin{equation}\label{eq:radialinit}
  \rho_0(\x)=\rho_0(r),\qquad
  \u_0(\x)=\frac{\x}{r}u_0(r),\qquad
  \phi_0(\x)=\phi_0(r).
\end{equation}
To ensure regularity of $\u$ and $\rho$ at the origin, we impose the following boundary conditions
\begin{equation}\label{eq:radialzero}
  \pa_r\rho(0,t)=0,\qquad
  u(0,t)=0,\qquad
  \pa_r\phi(0,t)=0.
\end{equation}
The initial density and potential are required to satisfy
\begin{equation}\label{eq:compatibility}
  \rho_0(\x)=\det\bigl(I-D^2\phi_0(\x)\bigr),
  \qquad \x\in\R^n.
\end{equation}
We refer to \eqref{eq:compatibility} as the
\emph{compatibility condition}. It expresses that $\rho_0$ and $\phi_0$ cannot be prescribed independently. We further restrict the initial potential by 
\begin{equation}\label{eq: s>0}
   \sup_{r>0} \phi''_0(r)< 1.
\end{equation}
Since $\phi_0'(0)=0$, we have
\[
\sup_{r>0}  \frac{\phi_0'(r)}{r} =\frac1r\int_0^r\phi_0''(s)\,\dd s  <1.
\]
Together with \eqref{eq:compatibility} and its radial form \eqref{eq:MAspectral}, the restrictions on $\phi$ imply $\inf_{\x>0} \rho_0(\x)>0$. Unless explicitly stated otherwise, the boundary and initial data conditions above are assumed throughout.

\subsection{Eulerian dynamics and the spectral dynamics} The damped EMA system belongs to the broad class of pressureless Eulerian dynamics
\begin{align*}
& \pa_t\rho+\div(\rho\u)=0,\\
& \pa_t(\rho\u)+\div(\rho\u\otimes\u)=\rho\F,
\end{align*}
where $\F$ collects the forcing terms. In our case, $\F := -\kappa \nabla \phi-\beta \u$ and $\phi$ is the potential that satisfies the Monge-Amp\'ere equation \eqref{eq:MA}. In the non-vacuum region, the momentum equation \eqref{eq:momentum} simplifies to
\begin{equation}\label{eq:velocity:non_vacuum}
\pa_t\u+(\u\cdot\grad)\u=\F.
\end{equation}
Under the assumption of radial symmetry, the force $\F$ is expressed as
\[
\F(\x,t) = -\kappa\frac{\x}{r}\pa_r\phi(r,t)- \beta \frac{\x}{r}u(r,t),
\]
which ensures the persistence of radial symmetry over time. Furthermore, let $U$ be defined as
\begin{equation}\label{eq:potential_u}
    U(\x,t):=\int_0^{|\x|} u(s,t) \ \d s.
\end{equation}
We observe that $\u(\x,t) = \grad U(\x,t)$, implying the no-swirl (irrotational) solutions persist in time as well.
\subsubsection{1D motivation and Riccati structure}
If $\F \equiv 0$, \eqref{eq:velocity:non_vacuum} reduces to the classical \textit{inviscid Burgers equation}. It is well known that the solution leads to finite-time shock formation for generic smooth initial data. Indeed, in one dimension, taking the spatial derivative of the equation yields 
$$(\partial_t + u \partial_x)(\partial_x u) = -(\partial_x u )^2.$$ 
Let $X(t)=X(t;x)$ denote the characteristics originating from $x\in\R$, defined by
\[\frac{\dd}{\dd t}X(t) = u\big(X(t),t\big),\quad X(0) = x.\]
Let $^\prime$ be the derivative along the characteristic path
\[f'(t): = \frac{\dd}{\dd t}f(X(t),t) = \pa_t f(X(t),t) + u(X(t),t)\,\pa_xf(X(t),t)).\]
we obtain $(\partial_xu(X(t),t))^\prime = -(\partial_x u(X(t),t))^2$. This is a Riccati type equation that encodes the solution structure: blowup occurs in a finite time if $\partial_x u_0(x) < 0$. The Riccati structure based analysis has been extended to various models, for example, see \cite{luan2026euler,bhatnagar2020critical2, kiselev2018global, tan2021eulerian, bhatnagar2023critical, do2018global, carrillo2016critical,choi2025critical, tadmor2014critical}.

\subsubsection{Spectral dynamics in Multidimensions} To extend the above approach to multidimensions, we take the spatial gradient of \eqref{eq:velocity:non_vacuum} and obtain
\begin{equation}\label{eq:gradu}
  (\pa_t+\u\cdot\grad)\grad\u+(\grad\u)^{2}=\grad\F.
\end{equation}
The uniform boundedness of the n-by-n velocity gradient matrix $\|\nabla \u(\cdot, t)\|_{L^\infty}$ is of particular importance, since in many cases, this quantity plays a significant role in establishing global regularity. 

A natural candidate for capturing the 1D Riccati structure in higher dimensions is the set of eigenvalues of the velocity gradient $\nabla \u$, denoted by $\{\lambda_i\}_{i=1}^n$. In the unforced case $\F\equiv 0$, the dynamics of $\lambda_i$, known as the \emph{spectral dynamics}, exhibits the same Riccati structure as $\partial_x u$. Indeed, we observe
\[
(\partial_t+ \u\cdot \nabla )\lambda_i = -\lambda_i^2, \quad i=1,\cdots,n,
\]
which can be explicitly solved along the characteristics path. This approach was thoroughly studied in \cite{liu2002spectral}. While the dynamics of $\{\lambda_i\}_{i=1}^n$ enjoys the explicit Riccati structure, in the presence of nontrivial forcing terms, the spectral dynamics of \eqref{eq:gradu} becomes
\begin{equation}\label{eq:spectral:F}
  (\pa_t+\u\cdot\grad)\lambda_i=-\lambda_i^2+\langle {\mathbf l}_i, (\grad\F) {\mathbf r}_i\rangle\quad i=1,\cdots,n,
\end{equation}
where $(l_i,r_i)$ are the corresponding left and right eigenvectors of $\lambda_i$. We face an immediate problem of controlling the term $\langle {\mathbf l}_i, (\grad\F) {\mathbf r}_i\rangle$. As the forcing gradient $\grad\F$ may not share the same eigen-system as $\grad\u$. 

However, in the radial symmetry setting, the two matrices turn out to share the same eigen-system (see Lemma \ref{eigensystem}). This allows for the study of the spectral dynamics associated with \eqref{eq:gradu}. More precisely, we will analyze the dynamics of the following quantities 
\[
\bigg(u_0'(|\x|), \,\,\frac{u_0(|\x|)}{|\x|}, \,\,\phi_0''(|\x|),\,\,
    \frac{\phi_0'(|\x|)}{|\x|}\bigg).
\]
In particular, as established in Lemma \ref{lem:gradu}, the uniform boundedness of these quantities are key to global regularity.
\subsubsection{The divergence approach}
An alternative multidimensional analogue to $\pa_x u$ is \emph{the divergence}
\[
d:= \div \u = \tr(\grad \u) = \sum_i^n \lambda_i.
\]
Taking the trace of the \eqref{eq:gradu} yields its dynamics:
\begin{equation}\label{eq: div}
    (\pa_t+\u\cdot\grad)d = -\tr\left((\grad\u)^{\otimes2}\right)+\div\F.
\end{equation}
While $d$ is real-valued and couples more easily to forcing terms, as $\div \F$ is typically simpler to control than $l_i^T(\grad\F) r_i$, the explicit Riccati structure is lost because $\tr\left((\grad\u)^{\otimes2}\right)\neq d^2$, for $n\geq2$. 
The discrepancy is governed by the \emph{spectral gap} of the matrix
$\grad\u$, defined as
\begin{equation}\label{eq:spectralgap}
  \eta =\frac{1}{2}\sum_{i=1}^n\sum_{j=1}^n(\lambda_i-\lambda_j)^2.
\end{equation}
This gives
\begin{equation}\label{eq:diff}
  d^2-\tr\left((\grad\u)^{\otimes2}\right)=\frac{n-1}{n}d^2-\frac{1}{n}\,\eta.
\end{equation}
To extend 1D regularity result and utilize the Riccati structure via the divergence $d$, one must additionally control the spectral gap, which is often a delicate and difficult task, see \cite{shvydkoy2019global}. For instance, one approach involves decomposing the velocity gradient $\grad \u$ into its symmetric and anti-symmetric part and controlling the spectral gaps of these two matrices individually, as demonstrated in \cite{he2017global}. Further discussions on the dynamics of the divergence $d$ can be found \cite[Remark 2.3, Remark 2.6]{tan2021eulerian}.

In the radially symmetric setting, the divergence simplifies to
\begin{equation}\label{eq:divrelation}
  d=\div\u=u_r+(n-1)\frac{u}{r}.
\end{equation}
Then we find that the discrepancy in \eqref{eq:diff} becomes
\begin{equation}\label{eq:sgrelation}
  d^2-  \tr\left((\grad\u)^{\otimes2}\right)=
  2(n-1)u_r\frac{u}{r}+(n-1)(n-2)\frac{u^2}{r^2}.
\end{equation}
Because $d$ is a weighted sum of $u_r$ and $u/r$, this term involves quadratic expressions of $u_r$ and $\frac{u}{r}$ which cannot be determined solely by the \emph{local information} of $d$. However, a key observation is that under the radial symmetry, both the divergence of the spectral gap can be fully determined by $u_r$ and $\frac{u}{r}$. 
In fact, the spectral gap simplifies greatly to $\eta=(n-1)(u_r-\frac{u}{r})^2$.
Furthermore, unwrapping the divergence evolution equation \eqref{eq: div} gives
\[
\Big(u_r+(n-1)\frac{u}{r}\Big)' = -u_r^2 - (n-1)\frac{u^2}{r^2}-\kappa \phi''-\kappa(n-1)\frac{\phi'}{r}-\beta u_r - (n-1)\beta \frac{u}{r},
\]
where $'$ is the derivative along the characteristic path. This structure suggests the importance of the pairs $(u_r, \frac{u}{r})$ and $(\phi'',\frac{\phi'}{r})$. We will see shortly in Lemma \ref{eigensystem} that they are precisely the eigenvalues of the velocity gradient $\grad \u$ and the potential Hessian $D^2\phi$, respectively. As it turns out, controlling these four quantities is key to global regularity. This leads to the first result of this paper.

\subsection{Local well-posedness}
We state the following local well-posedness and regularity criterion result for the damped EMA system.

\begin{theorem}[Local well-posedness]\label{thm: LWP}
Let $s > \frac{n}{2}$. Consider the damped EMA system \eqref{eqs:EMA} under the standing assumptions of Section ~\ref{sec:standing-assumptions}, suppose that
    \[
  \O_0(|\x|) := \left[u_0'(|\x|), \,\,\frac{u_0(|\x|)}{|\x|}, \,\,\phi_0''(|\x|),\,\,
    \frac{\phi_0'(|\x|)}{|\x|} \right]^\top\in \big(H^s(\R^n)\big)^4, \qquad s> \frac{n}{2}.
    \]
    Then there exists a time $T>0$ such that the solution $\O(\x,t)$ satisfies
     \begin{equation}\label{eq:regularity}
    \O(|\x|,t)\in C([0,T],  H^{s}(\R^n))^4.
  \end{equation}
  Moreover, the life span $T$ can be extended as long as
  \begin{equation}\label{eq:BKM}
    \int_0^T \|\O(\cdot,t)\|_{L^\infty}\,{\d}t<+\infty.
  \end{equation}
\end{theorem}
The statement of the theorem is a direct consequence of an a priori estimate. For the sake of completeness, we provide a proof of this result in Appendix \ref{sec:LWP}.

\subsection{Global regularity versus finite time blowup}
The next main result of this work concerns the phenomenon of critical thresholds. As implied by the regularity criterion \eqref{eq:BKM}, global regularity depends on the boundedness of $\O(\x,t)$. In this paper, we demonstrate that the damped EMA system \eqref{eqs:EMA} exhibits a \emph{sharp} and \emph{explicit} critical threshold. Specifically, subcritical initial data ensure $\O$ remain bounded for all time, leading to global well-posedness via Theorem \ref{thm: LWP}. On the contrary, supercritical initial data results in finite time singularity formation, where the velocity gradient and density lose uniform boundedness.

\begin{theorem}[Critical thresholds]\label{thm:thresholds}
Consider the damped EMA system \eqref{eqs:EMA} with initial data satisfying the assumptions of Theorem ~\ref{thm: LWP}, then the following characterization of the critical threshold holds:
 \begin{itemize}
  \item Global well-posedness: there exists a subcritical region $\CTsub\subset\R\times\R$, such that if the initial data satisfy
   \[\big(u_0'(r),\phi_0''(r)\big)\in \CTsub,\quad \forall~r>0,\]
   then the damped EMA system \eqref{eqs:EMA} has a unique global smooth solution, in the sense that \eqref{eq:regularity} holds for all time $T$.
  \item Finite time singularity formation: there exists a supercritical region $\CTsup = \CTsub^c$, such that if the initial data satisfy
  \[\exists~r_c>0 \quad\text{such that}\quad \big(u_0'(r_c),\phi_0''(r_c)\big)\in \CTsup,\]
   then smooth solution to the damped EMA system \eqref{eqs:EMA} only exists for $t\in [0, T_c)$ and moreover, the solution generates a singular shock at the blowup time $T_c$ and location $r_c=r(T_c; r_0)$, namely
   \begin{equation}\label{eq:singularshock}
	\lim_{t\to T_c^-}\pa_r u(r_c,t)=-\infty,\quad \lim_{t\to T_c^-}\rho(r_c,t)=\infty \ (\textnormal{or} \ 0).     
   \end{equation}
 \end{itemize}
\end{theorem}
We provide a more precise and explicit characterization of these threshold conditions via two equivalent constructions. The first, in Section 4, utilizes a direct phase plane analysis of the explicit solutions to the spectral dynamics. The second, in Section 5, employs a comparison principle based on a Lyapunov function approach. The following figure provides an illustration of the critical thresholds $\CTsub$ and $\CTsup$ in terms of $(p,\mu) = (\pa_ru, \pa_r^2\phi)$.

\begin{figure}[h!]
    \centering
    \begin{subfigure}[b]{0.33\textwidth}
        \centering
        \includegraphics[width=\linewidth]{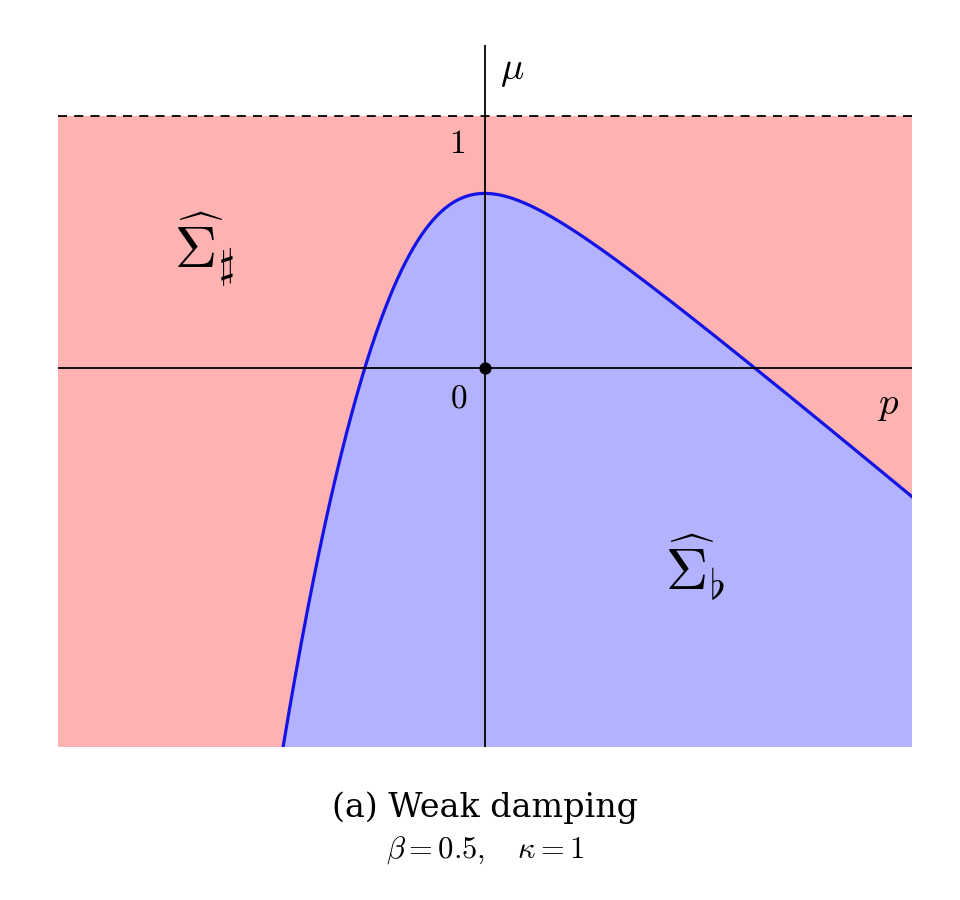}
    \end{subfigure}\hfill%
    \begin{subfigure}[b]{0.33\textwidth}
        \centering
        \includegraphics[width=\linewidth]{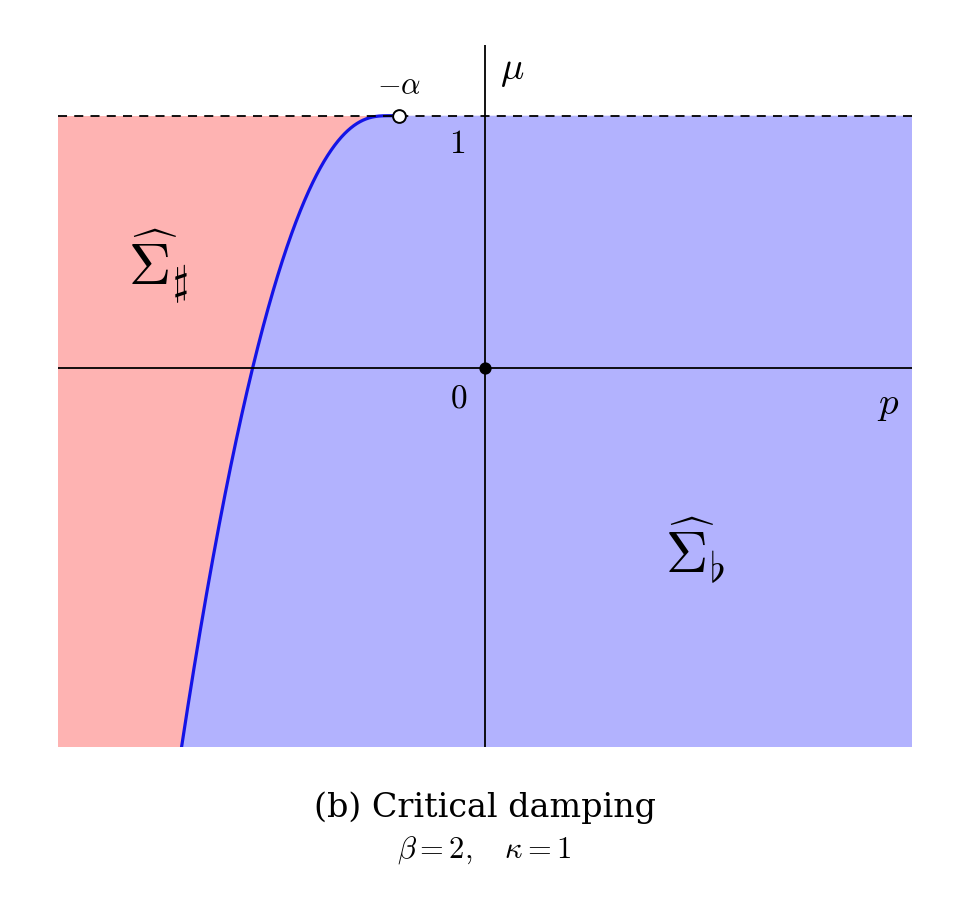}
    \end{subfigure}\hfill%
    \begin{subfigure}[b]{0.33\textwidth}
        \centering
        \includegraphics[width=\linewidth]{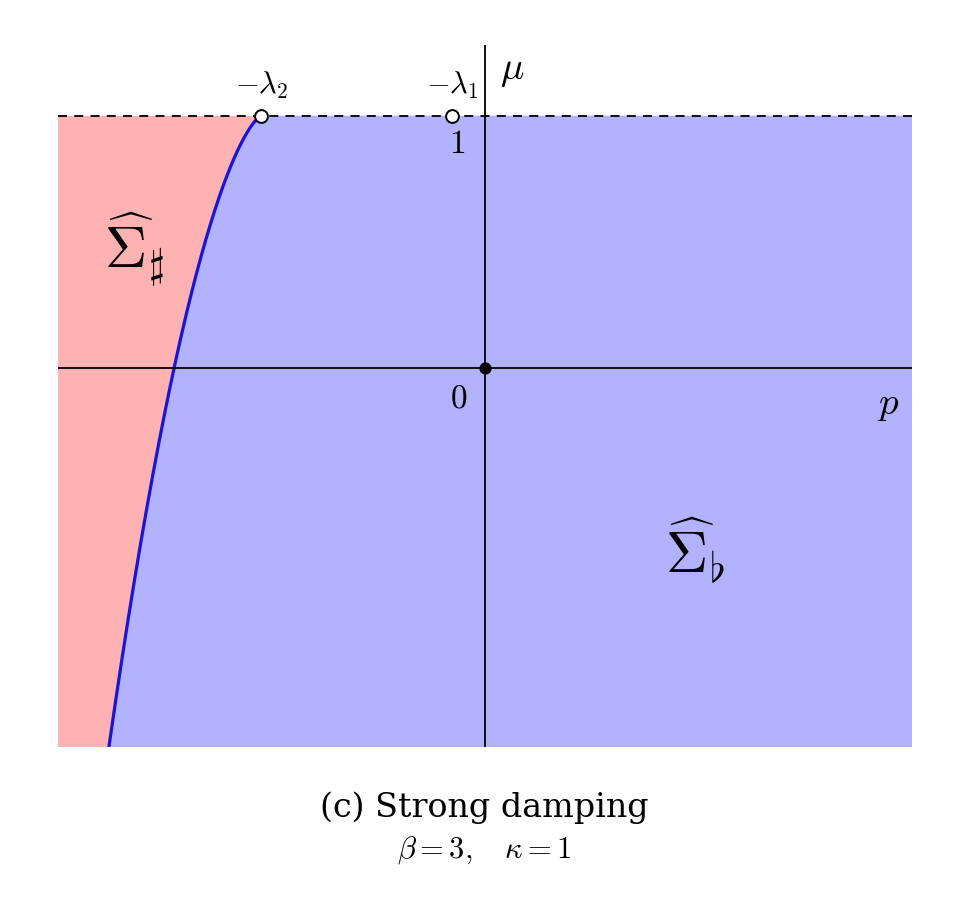}
    \end{subfigure}
    \caption{Phase portraits of the critical thresholds
             for the three damping regimes.}
    \label{fig:pmu-three-graphs}
\end{figure}

\subsection{Large-time behavior}
The large-time dynamics of pressureless Eulerian systems have been recently discussed in \cite{tadmor2025large, carrillo2025exponential, rozanova2024collision,rozanova2023behavior}. In particular, \cite{carrillo2025exponential} provides a rigorous and comprehensive analysis of systems coupled with nonlocal velocity alignment and various interaction forces. In this work, we establish the large-time behavior of the damped EMA system \eqref{eqs:EMA}, with optimal exponential rates of convergence. 

\begin{theorem}[Large Time Behavior]\label{thm:LTB_intro}
Let $(\rho, \mathbf{u}, \phi)$ be a global classical solution of \eqref{eqs:EMA}. Then there exists a constant $C>0$, depending on $\beta, \kappa, n$, the initial data and, in the critically damped case, $\eps$, such that the following exponential decay holds for all $t \geq 0$:
\begin{equation}
     \|\grad\u(\cdot, t)\|_{L^\infty} +  \|\rho(\cdot,t)-1\|_{L^\infty} + \| D^2\phi(\cdot, t, )\|_{L^\infty(\R^n)} \leq Ce^{-\gamma t}.
\end{equation}
Furthermore, if in addition there holds that $u_0, \phi'_0 \in L^\infty(\R_+)$, then 
\begin{equation}
     \|\u(\cdot, t)\|_{L^\infty} + \| \grad \phi(\cdot,t)\|_{L^\infty(\R^n)} \leq C e^{-\gamma t}.
\end{equation}
The precise dependence of the decay rate $\gamma$ on parameters $(\beta, \kappa)$, the initial data and $\eps>0$ is provided in Theorem \ref{thm:decay}.
\end{theorem}

\section{The damped Euler--Monge--Amp\`ere system with radial symmetry.}
In this section, we discuss the implications of radial symmetry and derive the spectral dynamics for the damped EMA system~\eqref{eqs:EMA}. 
We first establish that, under radial symmetry, the velocity gradient~$\nabla \mathbf{u}$ and the forcing gradient~$\nabla \mathbf{F}$ share the same eigen-system. More generally, we show that the Hessians of all radially symmetric scalar fields share the same eigensystem.

\begin{lemma}\label{eigensystem}
    Let $f:\mathbb{R}^n\to\mathbb{R}$ be a radial scalar function $f(\x)=f(r)$ with $r=|x|$ and $f\in C^2$.
Then for any $\x\not =\mathbf{0}$, we have
\[
D^2 f(x)=\frac{f'(r)}{r}\,I
+\Big(\f''(r)-\frac{f'(r)}{r}\Big)\frac{xx^\top}{r^2}.
\]
Moreover, $D^2 f(x)$ has eigenvalues

\smallskip\noindent 
  $\bullet \ \displaystyle \lambda_1(D^2\f)=\f''(r)$ associated with the eigenvector $\v_1=\x$;\newline
  $\bullet \ \displaystyle \lambda_i(D^2\f)=\frac{\f'(r)}{r}$,\  $i = 2,\dots,n$, associated with
   $\{\v_2,\cdots, \v_n\}$ which span $\{\x^\perp\}$.
\end{lemma}

\begin{proof}
Denoting $\pa_i = \frac{\pa}{\pa{x_i}}$, observe under the radial symmetry
\[
\pa_if(\x)=f'(r)\,\pa_i r=f'(r)\,\frac{x_i}{r}.
\]
Differentiate again to obtain
\begin{align*}
    \partial^2_{ij} f(\x)   &=f''(r)\,\frac{x_i x_j}{r^2} +f'(r)\Big(\frac{\delta_{ij}}{r} - \frac{x_i x_j}{r^3}\Big) \\
    &=\frac{f'(r)}{r}\,\delta_{ij} +\Big(f''(r)-\frac{f'(r)}{r}\Big)\frac{x_i x_j}{r^2}.
\end{align*}
Write in matrix form to obtain
\[
D^2 f(\x)=\frac{f'(r)}{r}\,I
+\Big(f''(r)-\frac{f'(r)}{r}\Big)\frac{\x\x^\top}{r^2}.
\]
We verify the eigen-system as stated in the lemma
\begin{align*}
   (\lambda_1I - D^2f(\x))\x &=  \bigg[ \bigg(f''(r)-\frac{f'(r)}{r}\bigg)I - \bigg(f''(r)-\frac{f'(r)}{r}\bigg)\frac{\x\x^T}{r^2} \bigg]\x\\
   &=\bigg(f''(r)-\frac{f'(r)}{r}\bigg)\x -  \bigg(f''(r)-\frac{f'(r)}{r}\bigg)\frac{\x\x^T\x}{r^2} = 0,\\
   (\lambda_2I-D^2f(\x))\v_i &=\bigg(f''(r)-\frac{f'(r)}{r}\bigg)\frac{\x \x^T\v_i}{r^2} = 0, \quad \forall \ i\in \{2,\dots,n\}.
\end{align*}
The last line follows from $\v_i \in \x^\perp$, for all $i\in \{2,\dots,n\}$. This completes the proof.
\end{proof}
Consequently, the velocity gradient and the forcing gradient
\[
\grad \u(\x,t) = D^2 U(r,t), \ \ \ \grad \F(\x,t) = -\kappa D^2 \phi(r) - \beta D^2U(r,t),
\]
share the same eigenvectors. Here, $U$ is defined in \eqref{eq:potential_u}. This fact allows us to utilize the explicit Riccati structure in the spectral dynamics \eqref{eq:spectral:F}. Hence, we propose to use the pair
\begin{equation}\label{defn:pq}
    p(r,t):= \pa_ru(r,t),\quad q(r,t):= \frac{u(r,t)}{r},
\end{equation}
as the multidimensional analogue of $\pa_x u$ in 1D. Similarly, we use
\begin{equation}\label{defn:munu}
    \mu(r,t):= \pa_r^2\phi(r,t),\quad  \nu(r,t):= \frac{\pa_r\phi(r,t)}{r},
\end{equation}
to denote the eigenvalues of the potential gradient $D^2\phi(r,t)$. 

Let $':=\pa_t+u(r,t)\pa_r$ denote differentiation along the characteristics paths, then the spectral dynamics \eqref{eq:spectral:F} becomes
\begin{equation}\label{eq:pq}
    \begin{cases}
        p' = -p^2-\kappa \mu - \beta p,\\
        q' = -q^2 - \kappa\nu - \beta q.
    \end{cases}
\end{equation}
A remarkable fact of studying the pair $(p,q)$ is that the boundedness of the radial derivative $p=\pa_ru(r,t)$ is sufficient to guarantee the boundedness of $\grad\u$, as demonstrated in the next lemma.
\begin{lemma}\label{lem:gradu}
 Consider the radial velocity field
 \eqref{eq:radial},\eqref{eq:radialzero},
 where $\displaystyle \u(\x,t)=\frac{\x}{r}u(r,t)$ and $u(0,t)=0$.   Then
  \begin{equation}\label{eq:gradup}
    \|\grad\u(\cdot,t)\|_{L^\infty}\leq\|\pa_ru(r,t)\|_{L^\infty}.
  \end{equation}
\end{lemma}
\begin{proof}
    To verify \eqref{eq:gradup} recall that $\grad\u(\x,t)$ is given by the radial Hessian
\[
  \grad\u(\x,t) = D^2U(r)=q(r,t)\,I +
    \big(p(r,t)-q(r,t)\big)\,\frac{\x\,\x^\top}{r^2},
    \]
    where $(p(r,t),q(r,t)) = (\pa_ru(r,t), \frac{u(r,t)}{r})$ is as defined in \eqref{eq:pq}. Observe $\grad \u$ thus defines a bilinear form. Let ${\mathbf w}$ be an arbitrary unit vector, then we have
    \begin{align*}
        \langle \grad\u(\x,t) {\mathbf w},{\mathbf w}\rangle &= q(r,t)\|\mathbf w\|^2 + (p(r,t)-q(r,t)) \langle \frac{\x\x^T}{r^2}\mathbf{w}, \mathbf{w}\rangle \\
        &=q(r,t) + (p(r,t)-q(r,t)) \frac{\|\x^T\mathbf{w}\|^2}{r^2}.
    \end{align*}
    Let $\theta = \frac{\|\x^T\mathbf{w}\|^2}{r^2}$, and observe $\theta \in [0,1]$. Then
    \begin{align*}
        | \langle \grad\u(\x,t) {\mathbf w},{\mathbf w}\rangle|  \leq (1-\theta)|q| + \theta |p| 
         \leq \max\{|p|,|q|\}.
    \end{align*}                              
Moreover, by \eqref{eq:radialzero}, $u(0,t)=0$ and hence
\[ |q(r,t)|=\frac{1}{r}\left|\int_0^r\pa_s u(s,t)\,{\d}s\right|\leq
  \|p(\cdot,t)\|_{L^\infty},\]
and the statement follows from the last two inequalities.
\end{proof}
In a similar fashion, we establish a lemma that provides a uniform bound on $\nu$. Combining this result with the spectral representation of the Monge--Amp\'ere equation \eqref{eq:MA} allows us to bound $\|\rho\|_{L^\infty}$ by the radial eigenvalue $\mu$ of the Hessian $D^2\phi$.
\begin{lemma}Let $\mu$ and $\nu$ be defined as in \eqref{defn:munu}, and assume the boundary regularity ~\ref{eq:radialzero}, then we have
\begin{equation}\label{eq:bound:nu}
    \|\nu(\cdot, t)\|_{L^\infty} \leq \|\mu(\cdot, t)\|_{L^\infty}.
\end{equation}
\end{lemma}
\begin{proof}
    We observe 
    \[
    |\nu(r,t)| = \frac{1}{r}|\pa_r \phi(r,t)| = \frac{1}{r}\bigg|\int_0^r \pa_r^2 \phi(s,t) \ \d s+\pa_r\phi(0,t)\bigg|.
    \]
    Since $\pa_r\phi(0,t) = 0$, equation \eqref{eq:bound:nu} holds.
\end{proof}

\subsection{Spectral dynamics for the radial damped Euler--Monge--Amp\'ere system.}
We have derived Riccati-type ODEs for the pair $(p,q)$ that involves $\mu$ and $\nu$. The goal of this section is to derive the evolution equations for $(\mu,\nu)$ to close the ODE systems.
To this end, we study the spectral representations of the Monge--Ampère equation \eqref{eq:MA} and the continuity equation \eqref{eq:density}. They determine the dynamics of $(\mu,\nu)$, which, together with the velocity gradient components $(p,q)$, govern the global regularity of the system, as stated in Theorem \ref{thm: LWP}.

\begin{lemma}[Spectral form of the Monge--Amp\`ere equation]
Let $(\rho, \phi)$ satisfy the assumptions set forth in Section ~\ref{sec:standing-assumptions}. Let $(\mu,\nu)$ be given as in \eqref{defn:munu}. Then the density has the representation:
    \begin{equation}\label{eq:MAspectral}
        \rho = (1-\mu)(1-\nu)^{n-1}.
    \end{equation}
Furthermore, there holds
  \begin{equation}\label{eq:rhonurelation}
        \partial_r \big(r^n (1-\nu)^n\big) = n r^{n-1} \rho.
    \end{equation}
\end{lemma}

\begin{proof}
   The first claim is straightforward to verify:
   \[
    \rho = \det(\Id-D^2\phi) = \prod_{i=1}^n(1-\lambda_i(D^2\phi))=(1-\mu)
  (1-\nu)^{n-1}.
   \]
Next, from the definition \eqref{defn:munu}, we observe the following relation between $\mu$ and $\nu$:
\[\mu = \pa_r(r\nu) = r\pa_r\nu + \nu.\]
This implies $1-\mu = 1-\nu - r\pa_r \nu$. Substituting this equation into \eqref{eq:MAspectral}, we obtain 
\[(1-\mu)(1-\nu)^{n-1}=(1-\nu)^{n}-r\pa_r\nu\,(1-\nu)^{n-1}
  =\frac{1}{nr^{n-1}}\,\pa_r\big(r^n(1-\nu)^n\big).\]
This completes the proof.
\end{proof}

\begin{remark}\label{rmk:density}
    Recall that we require the potential $\phi$ to satisfy the initial condition $\phi_{0}''(r) < 1$. This is a direct consequence of the spectral Monge--Amp\`ere equation \eqref{eq:MAspectral}. In the one-dimensional case, \eqref{eq:MA} reduces to $-\partial_{xx}\phi = \rho - 1$; thus, the physical requirement of positive initial density, $\rho_0 > 0$, is equivalent to $\phi_0''(r) < 1$. In higher dimensions, the relationship between the positivity of $\rho_0$ and the bounds on $\phi_0''(r)$ becomes dimension-dependent. Specifically, if $\phi_0''(r) = \mu_0 > 1$, then \eqref{eq:radialzero} implies that $\phi_0'(r)/r = \nu_0 > 1$ as well; consequently, the sign of $\rho_0$ depends on the parity of the dimension $n$. Furthermore, as discussed in Remark~\ref{rmk:initial}, the condition $\phi_0''(r) \geq 1$ yields universal blowup phenomenon. To ensure physical consistency, we enforce the condition $\phi_0''(r) < 1$.
\end{remark}

\begin{lemma}[Spectral form of the continuity equation]\label{lemma:spectral:continuity} Let $\rho(\x,t)$ and $\u(\x,t)$ be as given in \eqref{eq:radial} and \eqref{eq:radialzero}. Under radial symmetry, the evolution of the density \eqref{eq:density} becomes
\begin{equation}\label{eq:rho}
    \pa_t\rho+\pa_r(\rho u)=-\frac{(n-1)\rho u}{r}.
\end{equation}
\end{lemma}
\begin{proof}
    Observe by radial symmetry, we have
    \begin{align*}
        \div(\rho\u) &= \sum_{i=1}^n \pa_r(\rho u) \frac{x_i^2}{r}  + \sum_{i=1}^n \ \rho u(\frac{1}{r}-\frac{x_i^2}{r^3}) =\pa_r (\rho u)+ \frac{(n-1)\rho u}{r}.
    \end{align*}
    Hence the continuity equation \eqref{eq:density} becomes \eqref{eq:rho}.
\end{proof}

Let us discuss several consequences of the above two lemmas, in particular, the dynamics of $\nu$.

\begin{lemma}
Let $\rho=\rho(r,t)$ and $u=u(r,t)$ be smooth radial solutions of the continuity equation \eqref{eq:rho}
and define
\[
e(r,t) := \int_{0}^{r} s^{n-1}\rho(s,t) \d s.
\]
Let $' = \pa_t + u\pa_r$ be the derivative along the characteristic path. Then
\begin{equation}\label{eq:rnu}
    e' = \pa_t e + u\pa_r e = 0, \quad \textnormal{and} \quad \big(r(1-\nu)\big)' = 0,
\end{equation}
where $\mu$ and $\nu$ are defined as in \eqref{defn:munu}. Consequently
\[
\nu' = q(1-\nu),
\]
where $q$ is defined as in \eqref{defn:pq}.
\end{lemma}

\begin{proof}
Multiplying the continuity equation \eqref{eq:rho} by $r^{n-1}$ and rearranging terms yields
\begin{equation}\label{eq:aux}
    \partial_t(r^{n-1}\rho) + \partial_r(r^{n-1}\rho u) = 0.
\end{equation}
straightforward computation yields
\[
\pa_t e(r,t) = \int_0^r s^{n-1}\pa_t \rho(s,t)\d s, \quad
\pa_r e(r,t) = r^{n-1}\rho(r,t).
\]
Integrating \eqref{eq:aux} over $[0,r]$ to obtain
\[
\pa_t e(r,t) + r^{n-1}\rho(r,t) u(r,t) = 0.
\]
Let $r=r(t; r_0)$ be the characteristic path initiated at $r_0$, satisfying $r' = u(r,t)$, and $ r(0; r_0)=r_0.$
This gives $$\pa_t e + u \pa_r e =  e'(r,t)=0.$$ 
Next, we derive the dynamics of $\nu$. Observe from \eqref{eq:rhonurelation}, we get $r^n(1-\nu)^n=ne$. The conclusion that $e' = 0$ then implies that $\big(r(1-\nu)\big)'=0.$ Consequently, we have the dynamics of $\nu$
\begin{equation}\label{eq:nu}
  \nu' = \frac{r'}{r}(1-\nu) = \frac{u}{r}(1-\nu) = q(1-\nu).
\end{equation}
\end{proof}
Combining with the previously derived evolution equations \eqref{eq:pq} for $(p,q)$, we observe that the dynamics of $(q,\nu)$ forms a closed ODE system along the characteristic paths
\begin{equation}\label{eq:qnu}
  \begin{cases}
    q' = -q^2-\kappa \nu-\beta q,\\
    \nu' = q(1-\nu).
  \end{cases}
\end{equation}
Crucially, the dynamics of $(p,\mu)$ satisfies a system of the identical form. To see this, recall the spectral representation of the mass continuity equation \eqref{eq:rho}
    \begin{equation}\label{eq:rho_longer}
        \pa_t\rho + u\pa_r \rho = -\rho \pa_ru- \frac{(n-1)\rho u}{r}.
    \end{equation}
Along the characteristic path the dynamics of $\rho$ becomes
\begin{equation}\label{eq:rho:characteristic}
    \rho' = -\rho(p+(n-1)q).
\end{equation}
Furthermore, the spectral form of the Monge-Amp\`ere equation \eqref{eq:MAspectral} implies
\[\mu = 1-\frac{\rho}{(1-\nu)^{n-1}}.\]
This together with \eqref{eq:nu}, yields
\begin{align*}
  \mu' =& -\frac{\rho'(1-\nu)^{n-1}+(n-1)\rho(1-\nu)^{n-2}\nu'}{(1-\nu)^{2n-2}} =p(1-\mu).
\end{align*}
Therefore, the dynamics of $(p,\mu)$ also forms a closed ODE system along
the characteristic paths
\begin{equation}\label{eq:pmu}
  \begin{cases}
    p' = -p^2-\kappa \mu-\beta p,\\
    \mu' = p(1-\mu).
  \end{cases}
\end{equation}
As established by Lemma \ref{lem:gradu} and \ref{eq:bound:nu}, since the tangential pair $(q,\nu)$ and the radial pair $(p,\mu)$ are governed by the identical ODE systems, the global regularity of the damped EMA system is entirely determined by the dynamics of $(p,\mu)$. Consequently, the threshold conditions and asymptotic behavior results derived for the radial components apply directly to the tangential components.

\begin{figure}[h!]
    \centering
    \begin{subfigure}[b]{0.33\textwidth}
        \centering
        \includegraphics[width=\linewidth]{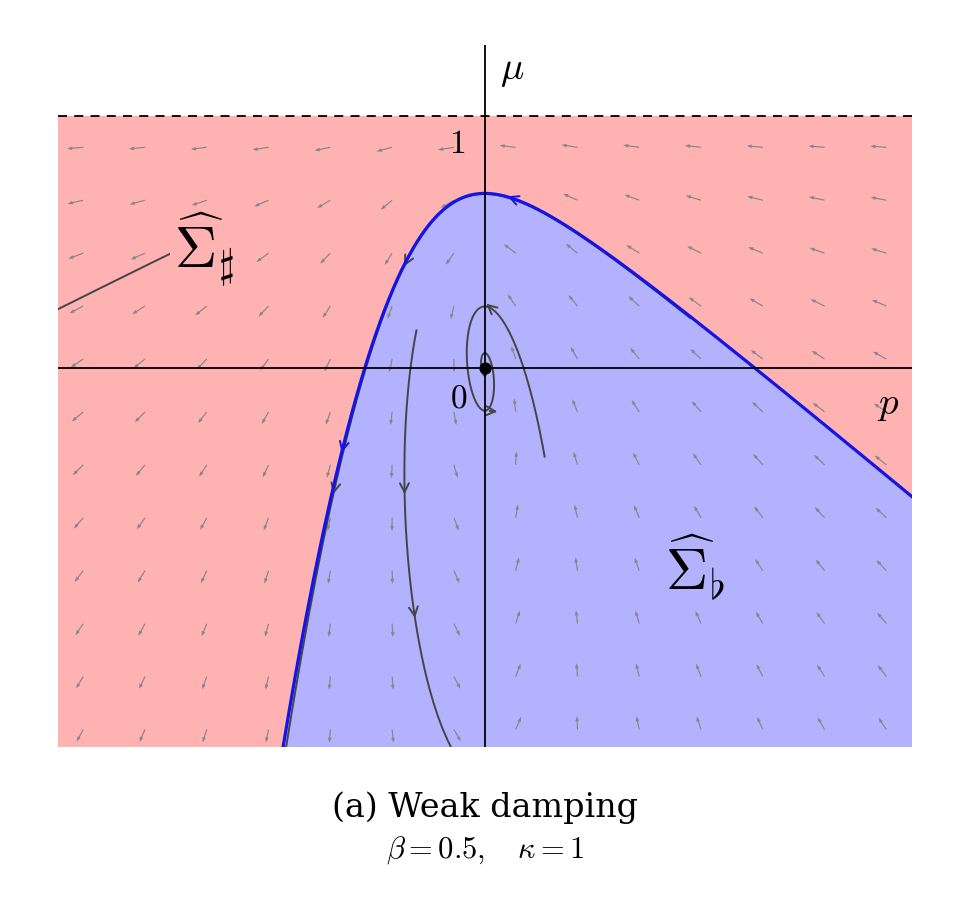}
    \end{subfigure}\hfill%
    \begin{subfigure}[b]{0.33\textwidth}
        \centering
        \includegraphics[width=\linewidth]{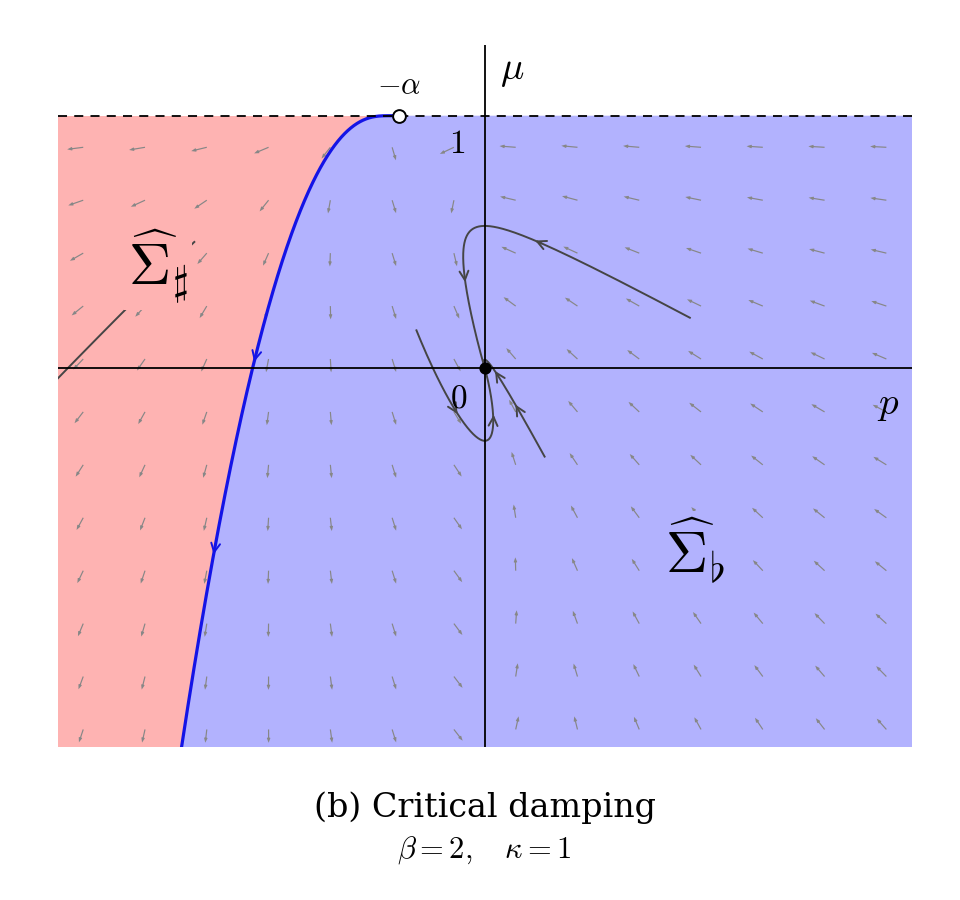}
    \end{subfigure}\hfill%
    \begin{subfigure}[b]{0.33\textwidth}
        \centering
        \includegraphics[width=\linewidth]{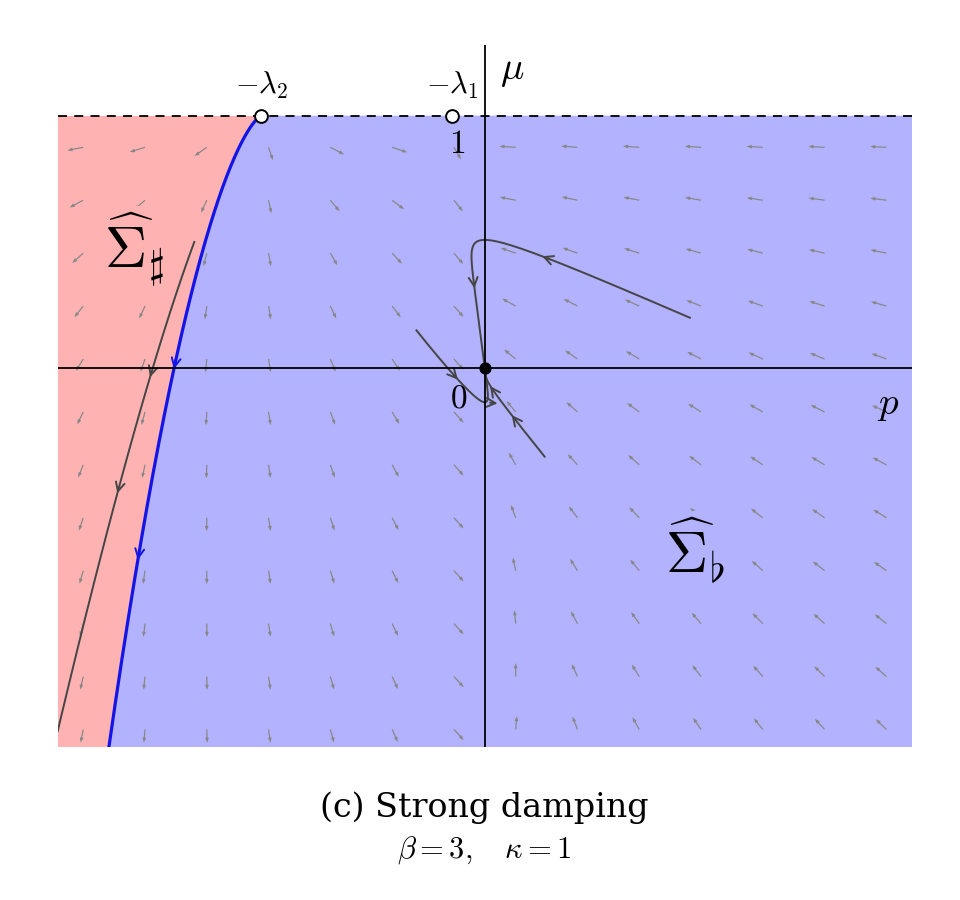}
    \end{subfigure}
    \caption{Phase portraits of the critical thresholds
             for the three damping regimes with particle trajectories.}
\end{figure}

\begin{remark}
    We would like to point out that one can also follow the analysis in \cite[Theorem 1, Theorem 4]{rozanova2024collision} to work with the dynamics of $(p,\mu)$ directly to study the problems of global regularity and large-time behavior.
\end{remark}

\section{Proof of the main results.}
In this section, we construct the sharp critical threshold regions in the $(p,\mu)$ plane, thereby characterizing the necessary and sufficient conditions for the uniform boundedness of $\grad \u$ and $\rho$ in the non-vacuum region. More specifically, we employ a transformation that reformulates \eqref{eq:pmu} as a system of linear ODEs, which enables us to conduct explicit phase plane analysis. This approach also yields the explicit exponential rate of decay to equilibrium, which establishes the system's large-time asymptotic behavior. An alternative construction using the Lyapunov functions is presented in the subsequent section.

\subsection{Construction of the critical threshold.}
The assumption $\phi_0''(r) = \mu_0 <1$ allows us to take advantage of the following two new variables.
\begin{equation}\label{transformation:ws}
    w:= \frac{p}{1-\mu}, \quad s:= \frac{1}{1-\mu}.
\end{equation}
The new pair of variables leads to a reformed coupled system, whose dynamics can be derived from \eqref{eq:pmu}:
\begin{align*}
    w^\prime &= \frac{p^\prime (1-\mu)+p \mu^\prime}{(1-\mu)^2} = \frac{(-p^2 -\kappa \mu - \beta p)(1-\mu)+p^2(1-\mu)}{(1-\mu)^2} = \kappa(1-s)-\beta w,\\
     s^\prime &= \frac{\mu^\prime}{(1-\mu)^2} = \frac{p}{1-\mu}= w.
\end{align*}
Hence, the dynamics of $(w,s)$ is given by
\begin{equation}\label{dynamics: w_s}
    \begin{cases}
        w^\prime = -\beta w +\kappa(1-s),\\
        s^\prime = w.
    \end{cases}
\end{equation}
Observe that the forcing terms in the right hand of \eqref{dynamics: w_s} are Lipschitz in $(w,s)$, which ensures that the solution $(w(t),s(t))$ remains bounded for all finite time. Furthermore, the solution $(w(t),s(t))$ decays in time; see the detailed computations in the following subsection. Since $p$ can be recovered via $p = \frac{w}{s}$, the only possible finite-time blowup scenario in the $(p,\mu)$ plane is when $s(T_c) = 0$ for some $T_c>0$. To characterize the conditions under which this occurs, we partition the $(w,s)$ phase plane into the following threshold regions:
\begin{itemize}\label{region}
    \item \textit{Subcritical region} $\Sigma_\flat$: if $(w_0,s_0) \in \Sigma_\flat$, then $s(t)>0$ for all time.
    \item \textit{Supercritical region} $\Sigma_\sharp$: if $(w_0,s_0)\in \Sigma_\sharp$, then there exists a finite time $T_c>0$ such that $s(T_c) = 0$.
\end{itemize}
The threshold regions in the $(p, \mu)$ plane can be recovered from those in the $(w,s)$ plane. Let us quickly review why such threshold regions exist in the $(w,s)$ plane in the undamped case ($\beta = 0$). We consider the following natural Lyapunov function, used in \cite{tadmor2022critical} to construct the critical threshold,
\begin{equation}\label{tt:lyapunov}
\mathcal{L}(w,s) = w^2 + \kappa(1-s)^2.    
\end{equation}
In this setting, $\frac{d}{dt}\mathcal{L}(w(t),s(t)) = 0$, and the trajectories of \eqref{dynamics: w_s} trace ellipses in the phase plane. The critical threshold is identified by the level set passing through the origin $(0,0)$, which corresponds to $\L(w(t),s(t)) =  \mathcal{L}(w_0,s_0) = \kappa$. Consequently, the condition
$s(t) > 0$ holds for all time provided that the initial data satisfy $\mathcal{L}(w_0, s_0) < \kappa$, or equivalently, $|p_0| < \sqrt{\kappa(1-2\mu_0)}$. 
Geometrically, the subcritical region in the undamped case is characterized as the interior region bounded by the ellipse centered at $(0,1)$ that intersects the origin. 

In Section 5, we introduce a more general family of Lyapunov functions, whose zero level sets coincide with the solution trajectories to the system \eqref{dynamics: w_s}. We then utilize a series of \textit{comparison principles} to construct the critical thresholds; see Proposition ~\ref{comparison_principle}. Next, we present the constructions of the critical threshold utilizing explicit solutions to \eqref{dynamics: w_s}.

\subsection{Explicit phase plane analysis.}
In our case, the introduction of the damping term renders integral curves to the system \eqref{dynamics: w_s} non-closed. We can explicitly construct the thresholds in $(w(t),s(t))$ plane by solving \eqref{dynamics: w_s}. To this end, consider the following system of ODEs
\[
    \begin{bmatrix}
        w \\[5pt]
        s-1
    \end{bmatrix}^\prime = \begin{bmatrix}
        -\beta & -\kappa \\
        1 & 0
    \end{bmatrix} \begin{bmatrix}
        w \\[5pt]
        s-1
    \end{bmatrix},
\]
subject to initial conditions
\[
s(0) = s_0 = \frac{1}{1-\mu_0}, \quad w(0) = w_0 = \frac{p_0}{1-\mu_0}.
\]
Let us denote
    \[
    \lambda_{1} := \frac{\beta - \sqrt{\beta^2 - 4\kappa}}{2},\quad  \lambda_{2} := \frac{\beta + \sqrt{\beta^2 - 4\kappa}}{2}.
\]
We observe $-\lambda_1$ and $-\lambda_2$ are the two eigenvalues of the coefficient matrix. If $\beta<2\sqrt{\kappa}$, then $\lambda_1$ and $\lambda_2$ are complex-valued, with
\[
\lambda_1 = \alpha-i\omega, \ \ \ \lambda_2 = \alpha+i\omega, \ \ \ \textnormal{where} \ \ \  \alpha = \frac{\beta}{2},  \ \ \ \textnormal{and} \ \ \  \omega:= \frac{1}{2}\sqrt{4\kappa-\beta^2}.
\]
For analysis convenience, we also use the following second-order ODE, obtained by differentiating $s(t)$,
\begin{equation}\label{eq:2nd order ODE:S}
    \begin{cases}
        s''(t)+\beta s'(t)+\kappa s(t) = \kappa,\\
        s(0) = s_0,  \ \  s'(0) = w_0.
    \end{cases}
\end{equation}

Throughout, we will use $s'(t)$ and $w(t)$ interchangeably, with the understanding that they denote the same function. There are three scenarios:

\textbf{Strong damping} $(\beta>2\sqrt{\kappa})$. The solution takes the following form
\begin{subequations}
    \begin{align}
     s(t) &= 1 + A_1 e^{-\lambda_1 t} + A_2 e^{-\lambda_2 t},\numberthis \label{eq:s_strong} \\
    w(t) &= -A_1 \lambda_1 e^{-\lambda_1 t} - A_2 \lambda_2 e^{-\lambda_2 t}. \numberthis \label{eq:w_strong}
    \end{align}
\end{subequations}
The coefficients $A_1,A_2$ are determined by the initial data:
\begin{equation}\label{eq:strong_coeff}
      A_1 = \frac{w_0 + \lambda_2 (s_0 - 1)}{\lambda_2 - \lambda_1}, \quad  A_2 = \frac{-w_0 - \lambda_1 (s_0 - 1)}{\lambda_2 - \lambda_1}.
\end{equation}
Recall by \eqref{dynamics: w_s}, $w(t) = s'(t)$.  We first observe $s(t)$ can have at most one extremum. Let $t^*$ be the time at which $s$ attains the extremum. Setting $s'(t^*)=0$, we find
\begin{equation}\label{eq:extremum-cond}
    e^{(\lambda_2 - \lambda_1) t^*} 
    = -\frac{A_2 \lambda_2}{A_1 \lambda_1}, \ \ \ \textnormal{and} \ \ \ t^* = \frac{1}{\lambda_2-\lambda_1}\ln\bigg(-\frac{A_2\lambda_2}{A_1\lambda_1}\bigg).
\end{equation}
Therefore, an extremum exists if and only if $A_1$ and $A_2$ have opposite signs. On the other hand, if there doesn't exist an extremum, then $s(t)$ is monotone on $[0,\infty)$ and converges to $1$. Hence $s(t)\geq \min \{s_0, 1\}$ for all $t\geq 0$. Suppose the extremum exists, then a second derivative calculation yields
\[
    s''(t) = A_1 \lambda_1^2 e^{-\lambda_1 t} + A_2 \lambda_2^2 e^{-\lambda_2 t}.
\]
At $t = t^*$,  \eqref{eq:extremum-cond} gives
\[
    s''(t^*) = -(\lambda_2 - \lambda_1) A_1 \lambda_1 e^{-\lambda_1 t^*}.
\]
We distinguish two cases based on the initial data:
\begin{enumerate}
    \item If $A_1>0$ and $A_2<0$, then $s(t^*)$ is a local maximum. In this case $s(t^*)$ is bounded below by $\min\{s_0,1\}$.
    \item If $A_1<0$ and $A_2>0$, then $s(t^*)$ is a local minimum. 
\end{enumerate}
Therefore, a violation of $s(t)>0$ can only occur if $s$ possesses a local minimum at some $t^*>0$. From \eqref{eq:extremum-cond} and \eqref{eq:w_strong}, the condition $t^*>0$ is equivalent to $w_0<0$, which further boils down to $p_0<0$. Thus, in the strong damping regime, a forward-time loss of positivity of $s(t)$ can occur only when
\[
A_1<0,\quad A_2>0,\quad  p_0<0.
\]
Observe \eqref{eq:2nd order ODE:S} asserts $s''(t^*)+\kappa s(t^*) = \kappa$, implying
\[
    s(t^*) = \frac{\kappa - s''(t^*)}{\kappa}.
\]
Consequently, the positivity of $s(t^*)$ is equivalent to the condition $s''(t^*) < \kappa$. On substituting the expression of $t^*$, we have
\[
-A_1\lambda_1 (\lambda_2-\lambda_1)\bigg(-\frac{A_2\lambda_2}{A_1\lambda_1}\bigg)^{-\frac{\lambda_1}{\lambda_2-\lambda_1}} < \kappa.
\]
Further simplification using $A_1<0, A_2>0$ gives
\[
\bigg[\frac{(\lambda_2-\lambda_1)(-A_1\lambda_1)}{\kappa}\bigg]^{\lambda_2} < \bigg[\frac{(\lambda_2-\lambda_1)A_2\lambda_2}{\kappa}\bigg]^{\lambda_1}.
\]
Substituting the expression of $A_1$ and $A_2$ into the inequality yields
\[
\bigg[-\frac{\lambda_1 w_0+\kappa(s_0-1)}{\kappa}\bigg]^{\lambda_2}
\;<\;
\bigg[- \frac{\lambda_2 w_0+\kappa(s_0-1)}{\kappa}\bigg]^{\lambda_1}.
\]
In addition, the conditions $A_1<0$ and $A_2>0$ require
\[
w_0+\lambda_2(s_0-1)<0, \qquad w_0+\lambda_1(s_0-1)<0.
\]
In terms of $(p_0, \mu_0)$, the critical threshold condition becomes
\begin{equation}\label{eq: cs_strong}
    \bigg[-\frac{\lambda_1 p_0 + \kappa \mu_0}{\kappa (1-\mu_0)}\bigg]^{\lambda_2}
\;<\;
\bigg[-\frac{\lambda_2 p_0 + \kappa \mu_0}{\kappa (1-\mu_0)}\bigg]^{\lambda_1}.
\end{equation}
given 
\begin{equation}\label{eq: cs_strong_assumption}
    \frac{p_0+\lambda_2\mu_0}{1-\mu_0}<0, \quad \frac{p_0+\lambda_1\mu_0}{1-\mu_0}<0, \quad p_0<0.
\end{equation}
To summarize, a violation of $s(t)>0$ in the strong damping regime, corresponding to a blowup scenario in the main system \eqref{eqs:EMA}, occurs if and only if the initial condition $(p_0,\mu_0)$ satisfies \eqref{eq: cs_strong_assumption} but fails to meet the inequality \eqref{eq: cs_strong}. In particular, \eqref{eq: cs_strong_assumption} simplifies to $p_0+\lambda_2\mu_0<0$ and $p_0+\lambda_1\mu_0 <0$ from \eqref{eq: s>0}.
\medskip

\textbf{Critical damping} $(\beta = 2\sqrt{\kappa})$.
The solution to the system \eqref{dynamics: w_s} takes the form
\begin{subequations}
    \begin{align*}
        s(t) &= 1 + \bigg[(s_0-1) + \Big(w_0+\alpha(s_0-1)\Big)t\bigg]e^{-\alpha t}, \numberthis \label{eq:s_borderline} \\
        w(t) &= \bigg[w_0-\alpha \Big(w_0+\alpha (s_0-1)\Big)t\bigg]e^{-\alpha t}. \numberthis \label{eq:w_borderline}
    \end{align*}
\end{subequations}
We observe that $s(t)$ admits at most one local extremum. If no such extremum exists, $s(t)$ decays monotonically in time to $1$, implying that $s(t)$ is bounded below by $\min\{s_0, 1\}$. Suppose instead that a local extremum exists at $t^*>0$. Setting \eqref{eq:w_borderline} to $0$ yields
\[
    t^* = \frac{w_0}{\alpha\big(w_0+\alpha (s_0-1)\big)}
\]
For $t^*>0$, the signs of $w_0$ and $w_0+\alpha(s_0-1)$ must coincide. Evaluating the second derivative at $t= t^*$ gives
\[
s''(t^*) = -\alpha \big(w_0+\alpha (s_0-1)\big)e^{-\alpha t^*}.
\]
We distinguish two cases based on the initial data:
\begin{enumerate}
    \item If $w_0+\alpha(s_0-1)>0$ and $w_0>0$, then $s(t^*)$ is a local maximum. In this case $s(t^*)$ is bounded below by $\min\{s_0,1\}$.
    \item If $w_0+\alpha(s_0-1)<0$ and $w_0<0$, then $s(t^*)$ is a local minimum. 
\end{enumerate}
To ensure $s(t^*)>0$ in the latter case, we require $s(t^*)>0$. Substituting $t^*$ into \eqref{eq:s_borderline}, we obtain
\begin{align*}
    s(t^*) &= 1+ \frac{w_0 + \alpha (s_0-1)}{\alpha}e^{-\alpha t^*}.
\end{align*}
The condition $s(t^*)>0$ therefore is equivalent to
\[
e^{-\alpha t^*}< -\frac{\alpha}{w_0 + \alpha(s_0 - 1)}.
\]
Substituting the expression for $t^*$ leads to the critical threshold inequality: 
\[
\frac{w_0}{w_0+\alpha (s_0-1)} > \ln\bigg[-\frac{w_0+ \alpha(s_0-1)}{\alpha}\bigg],
\]
provided that 
\[
w_0 <0  \ \ \ \textnormal{and} \ \ \ w_0+\alpha(s_0-1)<0.
\]
In terms of $p_0$ and $\mu_0$, this condition is expressed as:
\begin{equation}\label{eq: cs_critical}
    \frac{p_0}{p_0+\alpha \mu_0} > \ln\bigg[-\frac{p_0+\alpha \mu_0}{\alpha(1-\mu_0)}\bigg],
\end{equation}
provided $p_0<0$ and $p_0+\alpha\mu_0<0$, where we used the assumption $1-\mu_0>0$ again. To summarize, a violation of $s(t)>0$ in the critical damping regime, corresponding to a blowup scenario in the main system \eqref{eqs:EMA}, occurs if and only if the initial condition $(p_0,\mu_0)$ satisfies $p_0<0$ and $p_0+\alpha\mu_0<0$ but fails to meet the inequality \eqref{eq: cs_critical}.

\medskip

\textbf{Weak damping} $(\beta< 2\sqrt{\kappa})$. The solution takes the following form
\begin{subequations}
    \begin{align*}
      s(t) &= 1 + e^{-\alpha t}\left[ (s_0 - 1)\cos(\omega t)
        + \frac{w_0 + \alpha (s_0 - 1)}{\omega}\sin(\omega t)\right],\numberthis \label{eq: s_weak}\\
        w(t) & = e^{-\alpha t}\bigg[w_0\cos(\omega t)-\frac{\alpha w_0 +\kappa(s_0-1)}{\omega}\sin(\omega t)\bigg]. \numberthis \label{eq: w_weak}
\end{align*}
\end{subequations}
Since $s(t)$ decays to $1$, it suffices to identify the conditions where the first local minimum of $s(t)$ is positive. To this end, we set $s'(t)=w(t) =0$ and find
\[
    \tan(\omega t) = \frac{\omega w_0}{\alpha w_0 + \kappa (s_0 - 1)}.
\]
At $t = t^*$, the second derivative calculation gives
\[
s''(t^*)  = \frac{-\kappa \cos(\omega t^*)e^{-\alpha t^*}}{\alpha w_0 + \kappa(s_0-1)}\bigg[\big(w_0+\alpha(s_0-1)\big)^2 + (\kappa-\alpha^2)(s_0-1)^2\bigg].
\]
Recall $\kappa-\alpha^2 = \frac{1}{4}(4\kappa - \beta^2)>0$ holds in the weak damping regime. Consequently, to have that $s(t^*)$ is a local minimum requires
\[
\textnormal{sgn}\ (\cos(\omega t^*)) = - \textnormal{sgn} \ (\alpha w_0 + \kappa(s_0-1)).
\]
This gives us a sequence of minima. The first local minimum occurs at
\begin{equation}\label{eq:t_minima}
    \omega t^* = \beta_0 + \arctan\!\bigg[\frac{\omega w_0}{\alpha w_0 + \kappa (s_0 - 1)}\bigg],
\end{equation}
where
\begin{equation*}
\beta_0 =
\begin{cases}
0, & \alpha w_0+\kappa(s_0-1) < 0\ \text{and}\ w_0<0,\\[4pt]
\pi, & \alpha w_0+\kappa(s_0-1) > 0, \\[4pt]
2\pi, & \alpha w_0+\kappa(s_0-1) < 0\ \text{and}\ w_0>0.
\end{cases}
\end{equation*}To find the critical threshold condition, we  express $s(t)$ as
\begin{equation*}
s(t) = 1 + R e^{-\alpha t}\cos(\omega t - \psi),
\end{equation*}
where
\begin{align*}
R &= \sqrt{(s_0-1)^2 + \frac{(w_0+\alpha(s_0-1))^2}{\omega^2}}, \quad 
\tan\psi = \frac{w_0+\alpha(s_0-1)}{\omega (s_0-1)}.
\end{align*}
At the first local minimum $t^*$, set $s'(t^*) = 0$ and solve for $\cos(\omega t- \psi)$ gives
\begin{equation}\label{eq:s_min}
s(t^*) = 1 - \frac{\omega}{\sqrt{\omega^2+\alpha^2}}\;R\,e^{-\alpha t^*}.
\end{equation}
This implies the positivity of $s(t^*)$ is equivalent to
\[
\frac{\omega^2}{\omega^2+\alpha^2}\,R^2\,e^{-2\alpha t^*} < 1.
\]
Since $\omega^2+\alpha^2=\kappa$, the above condition becomes
\[
\big(w_0+\alpha(s_0-1)\big)^2+ \omega^2 (s_0-1)^2 <\ \kappa e^{2\alpha t^*},
\]
where $t^*$ is given as in \eqref{eq:t_minima}. We then obtain the critical threshold inequality in terms of $(p_0,\mu_0)$,
\begin{equation}\label{eq:cs_weak}
    \bigg(\frac{p_0+\alpha \mu_0}{1-\mu_0}\bigg)^2 +\omega^2\bigg(\frac{\mu_0}{1-\mu_0}\bigg)^2< \kappa e^{2\alpha t^*}.
\end{equation}

To summarize, a violation of $s(t)>0$ in the weak damping regime, corresponding to a blowup scenario in the main system \eqref{eqs:EMA}, occurs if and only if the initial condition $(p_0,\mu_0)$ fails to meet the inequality \eqref{eq:cs_weak}.

In conclusion, we have the following global regularity theorem.

\begin{theorem}[Global regularity vs. finite-time blow up]\label{thm: GWP}
Consider the damped EMA system \eqref{eqs:EMA} with assumptions set forth in Section ~\ref{sec:standing-assumptions}. Then the solution $(\rho,\u,\phi)$ develops a singular shock at a finite time $T_c$ and location $r_c=r(T_c; r_0)$, namely
   \begin{equation}
	\lim_{t\to T_c^-}\pa_r u(r_c,t)=-\infty,\quad \lim_{t\to T_c^-}\rho(r_c,t)=\infty,
    \end{equation}

if and only if there exists an $r>0$ such that the following inequalities are met:
 \begin{itemize}

 \item \textit{Strong damping} $(\beta >2\sqrt{\kappa})$
 \begin{align*}
     &\bigg[-\frac{\lambda_1 u_0'(r)+\kappa \phi''_0(r)}{\kappa(1-\phi''_0(r))}\bigg]^{\lambda_2} \geq \bigg[-\frac{\lambda_2 u_0'(r)+\kappa \phi''_0(r)}{\kappa(1-\phi''_0(r))}\bigg]^{\lambda_1},\ \ \textnormal{and}\\
       &\max\bigg\{u'_0(r)+\lambda_2 \phi''_0(r),u'_0(r)+\lambda_1 \phi''_0(r), u'_0(r)\bigg\}<0, 
 \end{align*}
 where  \[
    \lambda_{1} = \frac{\beta - \sqrt{\beta^2 - 4\kappa}}{2},\quad  \lambda_{2} = \frac{\beta + \sqrt{\beta^2 - 4\kappa}}{2}.
\]

\item \textit{Critical damping} $(\beta = 2\sqrt{\kappa})$. 
\begin{align*}
  &\frac{u_0'(r)}{u'_0(r)+\alpha \phi''_0(r)} \leq \ln\bigg[-\frac{u_0'(r)+\alpha \phi_0''(r)}{\alpha(1-\phi''_0(r))}\bigg],\ \  \textnormal{and}\\
       &\max\bigg\{u_0'(r),u_0'(r)+\alpha\phi_0''(r) \bigg\}<0, 
\end{align*}
where $\alpha = \frac{\beta}{2}$.

\medskip
\item \textit{Weak damping} $(\beta < 2\sqrt{\kappa})$.
\[
\bigg[\frac{u_0'(r)+\alpha \phi_0''(r)}{1-\phi_0''(r)}\bigg]^2+ \bigg[\frac{\omega\phi_0''(r)}{1-\phi_0''(r)}\bigg]^2 \geq \kappa e^{\beta t^*}, 
\]
where 
\begin{align*}
     t^* = \frac{1}{\omega} \bigg[\beta_0 + &\arctan\bigg(\frac{\omega u_0'(r)}{\alpha u_0'(r) +\kappa \phi_0''(r)}\bigg)\bigg], \quad  \omega = \frac{1}{2}\sqrt{4\kappa-\beta^2}, \quad \alpha = \frac{\beta}{2},
     \\[4pt]
\beta_0 &= 
\begin{cases}
0, & \alpha  u_0'(r) +\kappa \phi_0''(r) < 0\ \text{and}\ u_0'(r)<0,\\[4pt]
\pi, & \alpha u_0'(r) +\kappa \phi_0''(r) > 0 ,\\[4pt]
2\pi, & \alpha u_0'(r) +\kappa \phi_0''(r) < 0\ \text{and}\ u_0'(r)>0.
\end{cases}
\end{align*}
\end{itemize}
Otherwise, $\rho$ and $\grad \u$ remain uniformly bounded for all time.
\end{theorem}

\begin{proof}
The proof utilizes the qualitative analysis of the spectral dynamics along characteristic paths, as established through the auxiliary variables $(w(t), s(t))$ and the local well-posedness results in Theorem \ref{thm: LWP}.

For initial data $(\rho_0, u_0, \phi_0)$ satisfying the subcritical conditions, the auxiliary variable $s(t)$ remains strictly positive for all $t \geq 0$ along every characteristic path. Consequently, the quantities $p(t) = w(t)/s(t)$ and $\mu(t) = 1 - 1/s(t)$ remain uniformly bounded for all time. 

By Lemma \ref{lem:gradu}, the uniform boundedness of $p(\cdot, t)$ implies that $\|\nabla \mathbf{u}(\cdot, t)\|_{L^\infty} < \infty$ for all $t \geq 0$. Furthermore, utilizing the spectral representation of the density $\rho = (1-\mu)(1-\nu)^{n-1}$ and applying the bounds on $\|\nu(r,t)\|_{L^\infty}$ from Lemma \ref{eq:bound:nu}, we conclude that $\|\rho(\cdot, t)\|_{L^\infty}$ is controlled by the uniform bounds on $\mu$. Thus, global regularity of the damped EMA system follows from the regularity criterion \eqref{eq:BKM}.

Conversely, suppose the initial data is supercritical, satisfying the inequalities in Theorem \ref{thm: GWP} for some $r_0 > 0$. The corresponding trajectory in the $(w, s)$ phase plane intersects the line $s = 0$ at a finite time $T_c$. As $t \to T_c^-$, we observe the following:
  \begin{align*}
        \lim_{t\to T_c^-} \mu(t) &= 1- \lim_{t\to T_c^-} \frac{1}{s(t)} = - \infty.\\
         \lim_{t\to T_c^-} p(t) &=  \lim_{t\to T_c^-} \frac{s'(t)}{s(t)} =  \lim_{t\to T_c^-} \big(\log s(t)\big)' = -\infty.
    \end{align*}
     This implies the ODE system \eqref{eq:pq} with initial data $p(0) = u_0'(r_0), \mu(0) = \phi_0''(r_0)$ becomes unbounded in finite time $T_c$ at $r_c = r(T_c:r_0)$.
\end{proof}
\begin{remark}\label{rmk:lowerbound}
In the original analysis of the undamped EMA system \cite{tadmor2022critical}, a necessary condition for global regularity was the density lower bound $\rho_0\geq 2^{-n}$, where $n$ is the dimension. This restriction is derived from the threshold condition $|u_0'(r)| < \sqrt{\kappa(1-2\phi_0''(r))}$, which implicitly requires the potential to satisfy $\phi_0''(r) < 1/2$. Given the spectral representation of the density $\rho = (1-\phi'')(1-\phi'/r)^{n-1}$, the aforementioned lower bound on $\rho_0$ was obtained. 

With the introduction of velocity damping, the density lower bound is effectively removed in the strong and critical damping regimes. More precisely, for $\beta \geq 2\sqrt{\kappa}$, the subcritical class contains initial data with arbitrarily small positive density. Indeed, an immediate consequence of Theorem ~\ref{thm: GWP} is that in both the strong and critical damping regimes, the following condition
\[
u_0'(r) \geq -\frac{\beta}{2}, \quad r>0,
\]
is sufficient for global regularity.

In contrast, this relaxation does not extend to the weakly damped regime. Indeed, let $T = \pi/\omega$, where
$\omega=\frac12\sqrt{4\kappa-\beta^2}$, the explicit solution \eqref{eq: s_weak} gives
\[
s(T) =1-e^{-\alpha T}(s_0-1)>0, \quad \alpha = \beta/2.
\]
Consequently, $s_0<1+e^{\alpha T}=:s^\#$, and hence $1-\mu_0>1/s^\#$. It then follows directly that the same lower bound
holds for $1-\nu_0$. Therefore, a necessary condition for global
regularity is
\[
\rho_0(r)=(1-\mu_0(r))(1-\nu_0(r))^{n-1}
>
\Bigg(1+\exp\bigg(\frac{\beta\pi}
{\sqrt{4\kappa-\beta^2}}\bigg)\Bigg)^{-n}.
\]
This lower bound approaches $2^{-n}$ as $\beta\to0^+$ and
vanishes as $\beta\to2\sqrt{\kappa}^{\, -}$. In other words, in the weakly damped regime, the density cannot be too small for any initial data to maintain global regularity. 
\end{remark}
\begin{remark}\label{rmk:initial}
We observe from the above computations that the restriction $s_0 > 0$, or equivalently $1 - \mu_0 > 0$, is a natural requirement for the existence of smooth solutions, and this assumption does not limit the scope of our global regularity results.
\end{remark}
\begin{remark}
    As noted earlier, the damped EMA system reduces to the damped Euler-Poisson system in one dimension. The critical thresholds for this 1D case were previously established in \cite{bhatnagar2020critical2}. Our results in Theorem ~\ref{thm: GWP} are consistent with this prior work under their assumptions and recover their thresholds.
\end{remark}

\subsection{Formal vacuum solution of the spectral system}
We examine the formal vacuum solution on the level of the reduced spectral system. Consider the scenario $\phi_0''(r) = 1$, which corresponds to an initial vacuum $\rho_0 = 0$, as indicated in \eqref{eq:MAspectral}. From the evolution of $\mu$ in \eqref{eq:pmu}, we have
\[
\mu(t) = 1-(1-\mu_0)e^{-\int_0^t p(s)\ \d s}.
\]
Setting $\mu_0 =\phi_0''(r)= 1$ implies that $\mu(t) = 1$ for all $t>0$. Then in \eqref{eq:pmu}, the dynamics of $p$ simplifies to an autonomous Riccati ODE:
\begin{equation}\label{eq:p_vac}
  p' = -p^2 - \beta p - \kappa.  
\end{equation}
Depending on the damping coefficient $\beta$, we have three scenarios:
\begin{itemize}
    \item[(I)] \textit{Strong damping} $(\beta>2\sqrt{\kappa})$. In this regime, $p(t)$ has two fixed points $p_-<p_+<0$, corresponding to the two real roots of the polynomial  $f(p) = p^2 + \beta p+\kappa = 0$:
    \[
    p_- = \frac{-\beta -\sqrt{\beta^2 - 4\kappa}}{2},\quad \textnormal{and} \quad p_+ = \frac{-\beta +\sqrt{\beta^2 - 4\kappa}}{2}.
    \]
Hence for any initial data $p_0 \in (p_-, \infty)$, the solution $p(t)$ is bounded and asymptotically approaches the stable equilibrium $p_+$. Conversely, for any $p_0 < p_-$, the quadratic growth of the $-p^2$ term dominates, and $p(t)$ diverges to $-\infty$ in finite time $t = T_c$. To see that $T_c$ is finite, we integrate \eqref{eq:p_vac} and obtain
\begin{align*}
    T_c = \int_0^{T_c} \d t &= \int_{p_0}^{-\infty} \frac{\d p}{-p^2-\beta p - \kappa} \\ &= \int_{p_0}^{-\infty} -\frac{\d p}{(p-p_-)(p-p_+)} = \frac{1}{p_+-p_-}\ln\left|\frac{p_0-p_+}{p_0-p_-}\right|.
\end{align*}
Since $p_0< p_-<p_+$, $T_c$ is positive and finite. 
    \item[(II)] \textit{Critical damping} $(\beta = 2\sqrt{\kappa})$. In this case, the polynomial $f(p) = -p^2-\beta p - \kappa$ has one real root $p_* = -\beta/2 = -\sqrt{\kappa}$. Following a similar analysis, we conclude that if $ p_0\in [p_*,\infty)$, then $p(t)$ remains bounded and approaches the equilibrium $p_*$. For $p_0< p_*$, $p(t)$ approaches $-\infty$ in finite time. To see this, we observe
    \begin{align*}
        T_c =  \int_{p_0}^{-\infty} \frac{\d p}{-p^2-\beta p - \kappa} =  \int_{p_0}^{-\infty} \frac{\d p}{-p^2-2\sqrt{\kappa} p - \kappa} =  \int_{p_0}^{-\infty} \frac{\d p}{-(p+\sqrt{\kappa})^2} = -\frac{1}{p_0+\sqrt{\kappa}}.
    \end{align*}
Since $p_0< p_* = -\sqrt{\kappa}$, we see that $T_c$ is positive and finite.
    \item [(III)] \textit{Weak damping} $(\beta< 2\sqrt{\kappa})$. In this case, the polynomial $f(p) = -p^2-\beta p - \kappa$ doesn't have a real root. All initial data $p_0$ lead to blowup at a finite time $T_c$ determined by the following integral
    \begin{align*}
        T_c &=  \int_{p_0}^{-\infty} \frac{\d p}{-p^2-\beta p - \kappa} = \int_{p_0}^{-\infty} -\frac{\d p}{(p+\beta/2)^2+(\sqrt{\kappa - \beta^2/4})^2}\\ &= \frac{1}{\sqrt{\kappa-\beta^2/4}}\bigg[\arctan\bigg(\frac{p_0+\beta/2}{\sqrt{\kappa-\beta^2/4}}\bigg)+\frac{\pi}{2}\bigg].
    \end{align*}
\end{itemize}
  \begin{remark}   It is important to emphasize that the preceding calculation should be understood only as a formal statement for the reduced spectral dynamics, and it does not constitute a PDE result for the EMA system with vacuum initial data. A rigorous treatment of the EMA system with vacuum initial condition lies outside the scope of this work.\end{remark}

\subsection{Large-time behavior.}\label{sec:large_time}

Next, we use the explicit solutions to the system \eqref{dynamics: w_s} to study the large-time behavior of subcritical initial data. Throughout, we use the notation $A \lesssim B$ to mean $A\leq CB$ for some positive constant $C$ which depends on up to parameters $\beta, \kappa,$ dimension $n$, the initial data and in the case of critically damped regime, some small $\eps>0$. Since
\[(p,\mu) = \bigg(\frac{w}{s},\frac{s-1}{s}\bigg),\]
the decay of $(p,\mu)$ is determined by the decay of $(w,s-1)$, provided  $\|s^{-1}\|_{L^\infty}$ is finite uniform-in-time. The explicit solutions indicate that for any $T>0$ sufficiently large, $|s(t)-1|<\frac{1}{2}$ for all $t>T$. The sub-criticality guarantees $s(t)>0$ for all $t\in (0,T]$. It henceforth follows that $\|s^{-1}(t)\|_{L^\infty}<\infty$. For initial data, our standing assumption on $\mu_0$ in Theorem ~\ref{thm: LWP} implies the pointwise bound $\mu_0(r)<1$, by Sobolev embedding, holds uniformly. This allows us to absorb $\|s^{-1}\|_{L^\infty}$ into $C$.
Next, we remark that $(q,\nu)$ are controlled by the radial variables. Indeed, by Lemma \ref{lem:gradu} and \eqref{eq:bound:nu}
\[
\|q(\cdot, t)\|_{L^\infty}\leq\|p(\cdot, t)\|_{L^\infty}, \quad  \|\nu(\cdot, t )\|_{L^\infty} \leq \|\mu(\cdot, t )\|_{L^\infty}.
\]
Thus boundedness and decay of $(p,\mu)$ imply the corresponding boundedness and decay of $(q,\nu)$. 
The explicit solutions to \eqref{dynamics: w_s} yield
\[\|w(\cdot,t)\|_{L^\infty}+ \|s(\cdot,t)-1\|_{L^\infty}\lesssim e^{-\gamma t},
\]
where $\gamma$ depends on the damping regime and will be specified below. Consequently, using the radial bounds for \(q\) and \(\nu\), we obtain
\[
\|p(\cdot,t)\|_{L^\infty}+\|q(\cdot,t)\|_{L^\infty}+\|\mu(\cdot,t)\|_{L^\infty}+\|\nu(\cdot,t)\|_{L^\infty}
\lesssim e^{-\gamma t}.
\]
Next, from the spectral representation of $\rho$ in \eqref{eq:MAspectral}
\begin{align*}
\rho = (1-\nu)^{n-1}(1-\mu)
  &= 1 - \mu - (n-1)\nu + R(\mu,\nu),
\end{align*}
where the remainder term $R$ consists of all higher order terms of $\mu$ and $\nu$. Thus, the same decay rate holds for $(\rho-1)$:
\[
\|\rho(\cdot,t)-1\|_{L^\infty}\lesssim e^{-\gamma t}.
\]
Moreover, since $(p,q)$ are the eigenvalues of $\grad \u$ and
$(\mu,\nu)$ are the eigenvalues of $D^2 \phi$, we conclude
\[
\| \grad \u(\cdot,t)\|_{L^\infty} + \|\rho(\cdot,t)-1\|_{L^\infty} + \|D^2\phi(\cdot,t)\|_{L^\infty} \lesssim e^{-\gamma t}.
\]
The preceding estimate gives decay of the spectral quantities $q = u/r$ and $\nu=\pa_r\phi/r$, but not directly of $u$ and $\pa_r \phi$. If, in addition $(u_0, \phi'_0) \in (L^\infty(\R^n))^2$, then $(\u, \nabla \phi)$ decay with the same rate. To see this, let $r(t;r_0)$ be the characteristic satisfying 
\[
r' = u(r(t),t), \quad r(t=0;r_0) = r_0.
\]
For $r_0>0$, by subcriticality of the initial data, we have
\[
r(t)=r_0\exp\left(\int_0^t q(r(\tau),\tau)\,d\tau\right)>0, \quad \forall t>0.
\]
The case $r_0 = 0$ is handled separately by the boundary conditions \eqref{eq:radialzero}. Consider $q(t) =q(r(t),t), \nu(t)= \nu(r(t),t)$ along the characteristic and define
\[
X(t):=r(t)q(r(t),t)=u(r(t),t),\qquad Y(t):=r(t)\nu(r(t),t)=\pa_r\phi(r(t),t).
\]
Using \eqref{eq:qnu}, we obtain
\[
X'=-\beta X-\kappa Y,\qquad Y'=X.
\]
Hence $Y$ solves the same ODE as \eqref{eq:2nd order ODE:S}
\[
Y''+\beta Y'+\kappa Y=0, \quad Y(0) = \phi_0'(r_0), \quad Y'(0) = u_0(r_0).
\]
We then have the same decay estimate
\[
|u(r(t;r_0),t)| + |\partial_r\phi(r(t;r_0),t)| \lesssim e^{-\gamma t} \left( |u_0(r_0)| +
|\phi'_0(r_0)| \right).
\]
Taking the supremum over $r_0$ yields
\[ \|\mathbf u(\cdot,t)\|_{L^\infty} + \|\nabla\phi(\cdot,t)\|_{L^\infty} \lesssim e^{ -\gamma t} \left( \|u_0\|_{L^\infty}  + \|\phi'_0\|_{L^\infty} \right).
\]
\begin{remark}
    It should be noted that the additional assumption on the uniform boundedness of $(u_0', \phi_0'')$ is compatible with similar existing work, see e.g., \cite{rozanova2024collision}.
\end{remark}
Next, we investigate the precise decay rate $\gamma$ in each damping regime.
\vspace{1mm}
\newline
\textbf{Strong damping} $(\beta > 2\sqrt{\kappa})$. From \eqref{eq:s_strong} and \eqref{eq:w_strong}, we have
\begin{equation*}
   \gamma=  \begin{cases}
        \lambda_1,\ \ \  \textnormal{if}  \ \ \ A_1\not =0 , \quad \textnormal{for some r} >0,\\
         \lambda_2, \ \ \ \textnormal{if} \ \ \ A_1= 0, \quad \textnormal{for all r} >0.
    \end{cases}
\end{equation*}
Here, $A_1$ is the solution coefficient defined in \eqref{eq:strong_coeff}. In terms of the primary variables,  
\begin{equation}
    \gamma = \begin{cases}
       \lambda_1 , \ \ \ \textnormal{if} \ \ \  u'_0(r)+\lambda_2 \phi''_0(r) \not=0, \quad \textnormal{for some r} >0,\\
       \lambda_2, \ \ \ \textnormal{if} \ \ \  u'_0(r)+\lambda_2 \phi''_0(r) =0, \quad \textnormal{for all r} >0.
    \end{cases}
\end{equation}
\textbf{Critical damping} $(\beta = 2\sqrt{\kappa})$. We observe the following holds
\[
te^{-\alpha t} \leq C_\eps e^{-(\alpha -\eps)t},
\]
for every $\eps>0$ and $C_\eps$ is a positive constant that depends on $\eps$. This implies
\begin{equation}
    \gamma =\begin{cases}
        \alpha-\eps, \ \ \ \textnormal{if} \ \ \  u'_0(r)+\alpha \phi''_0(r) \not=0, \quad \textnormal{for some r} >0,\\
        \alpha, \ \ \ \textnormal{if} \ \ \  u'_0(r)+\alpha \phi''_0(r) =0, \quad \textnormal{for all r} >0.
    \end{cases}
\end{equation}
\textbf{Weak damping} $(\beta<2\sqrt{\kappa}).$ The oscillatory solutions \eqref{eq: s_weak} and \eqref{eq: w_weak} decay according to the real part, which gives
\[
\gamma = \alpha.
\]

The following theorem summarizes the exponential decay rates for the damped EMA system \eqref{eqs:EMA} across all damping regimes.
\begin{theorem}\label{thm:decay}
    Let $(\rho,\u,\phi)$ be a global classical solution to the damped EMA system \eqref{eqs:EMA} with radial symmetry. Let $\eps>0$ be arbitrary. Then there
exists a constant $C>0$, depending on $\beta, \kappa, n$, the initial data and, in the critically damped case, $\eps$, such that 
   \begin{equation}\label{eq:decay_rate}
        \|\grad\u(\cdot, t)\|_{L^\infty} +\| \rho(\cdot,t)-1\|_{L^\infty} + \| D^2 \phi(\cdot,t)\|_{L^\infty} \le C e^{-\gamma t},
    \end{equation}
where the decay constant $\gamma$ is determined by the damping coefficient $\beta$ and the initial condition:
    \begin{itemize}
         \item \textit{Strong damping} $(\beta >2\sqrt{\kappa})$. 
        \begin{equation*}\gamma=
            \begin{cases}
                \frac{\beta-\sqrt{\beta^2-4\kappa}}{2},\quad \textnormal{if} \quad u'_0(r)+\lambda_2 \phi''_0(r) \not=0,\quad \textnormal{for some r} >0\\ 
                 \frac{\beta+\sqrt{\beta^2-4\kappa}}{2} \quad \textnormal{if} \quad u'_0(r)+\lambda_2 \phi''_0(r) =0 \quad \textnormal{for all r} >0.
            \end{cases}
        \end{equation*}
\item \textit{Critical damping} $(\beta = 2\sqrt{\kappa})$.   \begin{equation*}\gamma=
           \begin{cases}
        \frac{\beta}{2}-\eps , \ \ \ \textnormal{if} \ \ \  u'_0(r)+\alpha \phi''_0(r) \not=0,\quad \textnormal{for some r} >0\\
        \frac{\beta}{2}, \ \ \ \textnormal{if} \ \ \  u'_0(r)+\alpha \phi''_0(r) =0 \quad \textnormal{for all r} >0.
    \end{cases}
        \end{equation*}

\medskip
\item \textit{Weak damping} $(\beta < 2\sqrt{\kappa})$.
\[
\gamma = \frac{\beta}{2}.
\]
    \end{itemize}
    If in addition,
\[
u_0,\partial_r\phi_0\in L^\infty(\mathbb R^n),
\]
then the following variables also decay with the same rate:
\[
\|\u(\cdot,t)\|_{L^\infty} +\|\grad \phi(\cdot,t)\|_{L^\infty}
\le C e^{-\gamma t}.
\]
\end{theorem}
\begin{remark}
    The faster-rate conditions above are stated only in terms of $(p_0,\mu_0)$, but with the boundary conditions \eqref{eq:radialzero}, the corresponding tangential conditions follow directly.
\end{remark}

\section{A Lyapunov function based approach to critical thresholds}
In this section, we provide an alternative construction of the critical threshold, represented as part of the solution trajectory to the system \eqref{dynamics: w_s} satisfying certain properties. Let $C$ be a trajectory of \eqref{dynamics: w_s} starting at $ (0,0)$ and traces backwards in time, and define $R_1:=\{(w,s)\ : \ w<0, s > 0\}$. 

\begin{itemize}
    \item [(I)]  If $\beta \geq 2\sqrt{\kappa}$, then the trajectory $C$ will not exit the region $R_1$ (see explicit expressions of $C$ in \eqref{eq:strong}, \eqref{eq:medium}, and discussions in Lemma \ref{lem: lifespan}). We define the subcritical region $\Sigma_\flat$ as the open set in $\R\times \R_+$ to the right of the curve $C$.
    \medskip
    \item [(II)] If $\beta < 2\sqrt{\kappa}$, the trajectory $C$ will exit the region $R_1$ at $(0,s^*)$ and intersects $\{s=0\}$ at $(w^*,0)$. The value of $s^*$ is specified in \eqref{eq: s*}. We define the subcritical region $\Sigma_\flat$ as the open set enclosed by the portion of $C$ joining $(0,s^*)$ and  $(w^*,0)$, and the axis $\{s = 0\}$. 
\end{itemize}

\begin{figure}[h!]
    \centering
    \begin{subfigure}[b]{0.32\textwidth}
        \centering
        \includegraphics[width=\linewidth]{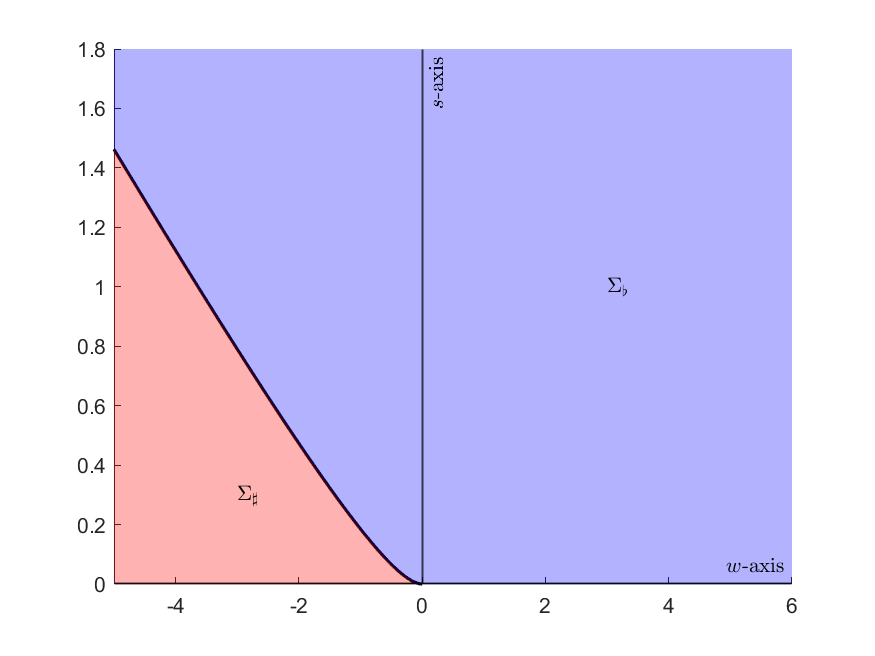}
        \caption{Strong Damping}
        \label{fig:weak}
    \end{subfigure}
    \hfill 
    \begin{subfigure}[b]{0.32\textwidth}
        \centering
        \includegraphics[width=\linewidth]{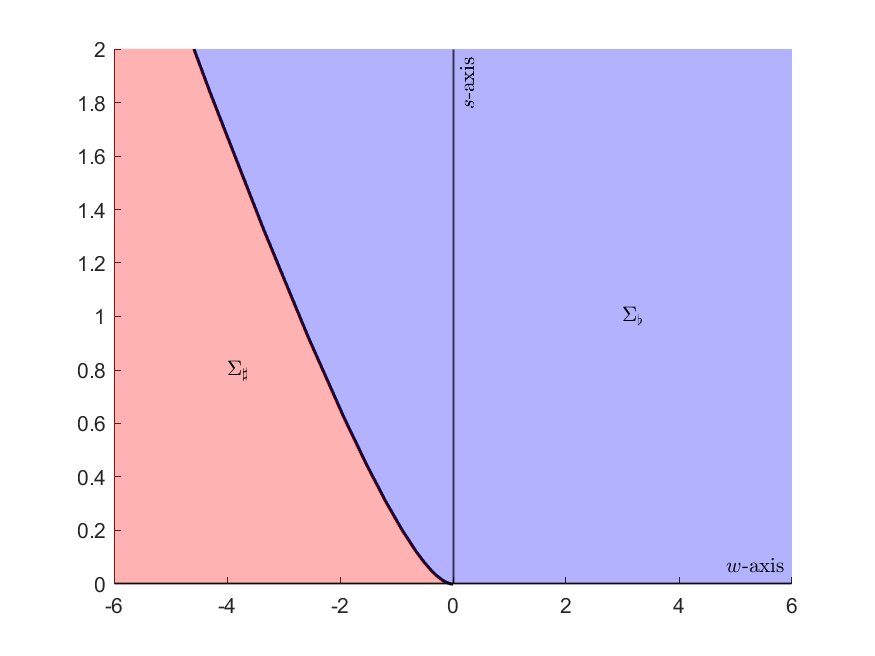}
        \caption{Critical Damping}
        \label{fig:critical}
    \end{subfigure}
    \hfill 
    \begin{subfigure}[b]{0.32\textwidth}
        \centering
        \includegraphics[width=\linewidth]{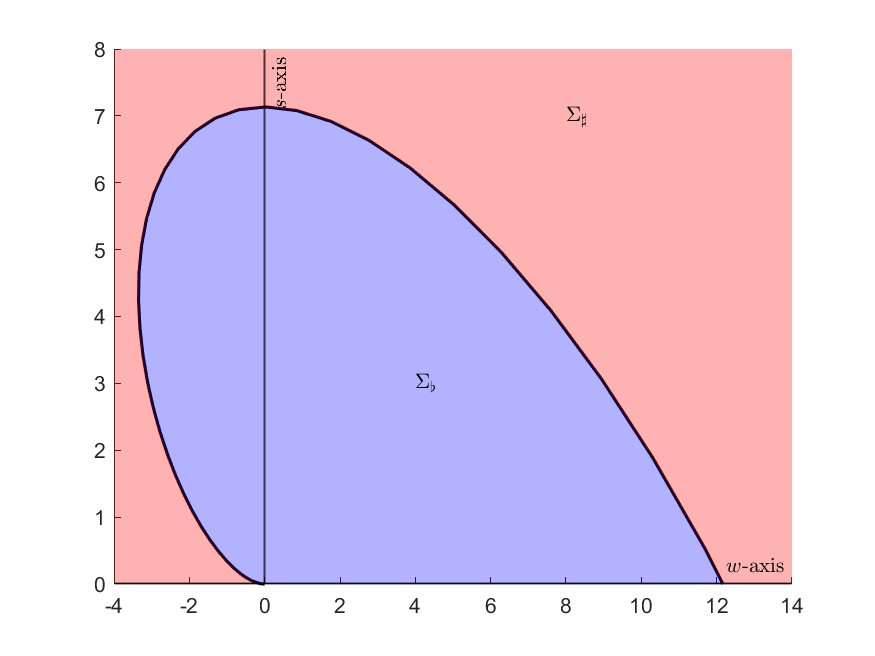}
        \caption{Weak Damping}
        \label{fig:strong}
    \end{subfigure}
    \caption{Phase portraits of the critical thresholds for the three damping regimes.}
    \label{fig:three_graphs}
\end{figure}

Solving for \eqref{dynamics: w_s} explicitly yields the following representations of $C$:

\begin{itemize}
    \item [(I)] \textit{Strong damping} $(\beta > 2\sqrt{\kappa})$. The solutions $(w(t),s(t))$ are given as
    \begin{subequations}\label{eq:strong}
\begin{align}
    s(t) &= -\frac{\lambda_2}{\lambda_2-\lambda_1}e^{-\lambda_1 t}+ \frac{\lambda_1}{\lambda_2-\lambda_1}e^{-\lambda_2 t}+1, \label{s:strong}\\
            w(t) &= \frac{\lambda_1\lambda_2}{\lambda_2-\lambda_1}(e^{-\lambda_1t}-e^{-\lambda_2t}) \label{w:strong},
\end{align}
where   \[
    \lambda_{1} = \frac{\beta - \sqrt{\beta^2 - 4\kappa}}{2},\quad  \lambda_{2} = \frac{\beta + \sqrt{\beta^2 - 4\kappa}}{2}.
\]
\end{subequations}	
Let $C_s$ be a portion of this trajectory defined as 
\begin{equation}\label{eq: C_s}
    C_s:= \Big\{(w(t),s(t)) \ : \ t\in (-\infty,0], \ \big(w(0),s(0)\big) = (0,0) \Big\}.
\end{equation}
    \item [(II)] \textit{Critical damping} $(\beta = 2\sqrt{\kappa})$. The solution $(w(t),s(t))$ is given as
     \begin{subequations}\label{eq:medium}
\begin{align}
    s(t) &=-e^{-\alpha t}-\alpha te^{-\alpha t}+1, \label{s:medium}\\
            w(t) &= \alpha ^2te^{-\alpha t} \label{w:medium}.
\end{align}
Here $\alpha =\beta/2$.
\end{subequations}	
Let $C_c$ be a portion of this trajectory defined as
\begin{equation}\label{eq: C_b}
    C_c:= \Big\{(w(t),s(t)) \ : \  t\in (-\infty,0], \ \big(w(0),s(0)\big) = (0,0) \Big\}.
\end{equation}

\medskip
    \item [(III)] \textit{Weak damping} $(\beta < 2\sqrt{\kappa})$. The Solution $(w(t),s(t))$ is given as
         \begin{subequations}\label{eq:weak}
\begin{align}
    s(t) &=\bigg[-\cos(\omega t)-\frac{\alpha}{\omega}\sin(\omega t)\bigg]e^{-\alpha t}+1, \label{s:weak}\\
            w(t) &= e^{-\alpha t}\bigg(\frac{\kappa}{\omega}\bigg)\sin(\omega t). \label{w:weak}
\end{align}
\end{subequations}	
Here $\omega = \frac{1}{2}\sqrt{4\kappa - \beta^2}$. Let $\tau: = \sup \{t<0 :\ s(t) = 0\}$ be the first negative time at which $s(\tau) = 0$. Let $C_w$ be a portion of this trajectory defined as
\begin{equation}\label{eq:weak:trajectory}
    C_w := \Big\{(w(t),s(t)) \ : \ t\in [\tau,0], \ \big(w(0),s(0)\big) = (0,0)\Big\}.
\end{equation}
\end{itemize}

The trajectories $C_s, C_c$ and $C_w$ constructed above define the critical threshold curves for their respective damping regimes. A more precise construction of the super- and subcritical regions utilizing the comparison principles will be provided. Note that the critical threshold in our case is \textit{sharp}, hence the supercritical region $\Sigma_\sharp$ is taken to be the complement of the subcritical regions, i.e., $\Sigma_\sharp: =\R\times\R_+ \setminus \Sigma_\flat$.

\begin{remark}
In the construction of subcritical regions for the weakly damped regime, the curve $C$ corresponds precisely to the solution trajectory $C_w$ defined in \eqref{eq:weak:trajectory}. The region $\Sigma_\flat$ is well-defined as the trajectory of \eqref{dynamics: w_s} is expanding when traced in negative time. Since $(w^*,0)$ represents the intersection of $C_w$ with $\{s=0\}$ at the first negative time, it follows that $w^* > 0$.
\end{remark}

\begin{remark}
   This construction is motivated by the fact that each integral curve in the phase plane either remains in the region $\{s>0\}$ for all time, or else reaches the boundary $\{s=0\}$ at some finite time $T_c>0$.

In the latter case, the intersection with $\{s=0\}$ can occur transversely or
tangentially. If the curve touches $\{s=0\}$ tangentially at some time $T_c$, then
$s(T_c)=0$ and $s'(T_c)=0$. By \eqref{dynamics: w_s}, this can happen only at
$(w,s)=(0,0)$. Poincar\'e's uniqueness theorem for autonomous systems implies that distinct integral curves cannot intersect. Hence, the trajectory through $(0,0)$ partitions the family of integral curves of \eqref{dynamics: w_s} into those that remain in $\{s>0\}$ and those that intersect $\{s=0\}$. The following figure depicts the phase plane velocity field and three distinct integral curves.

\begin{figure}[ht!]
    \centering
     \includegraphics[width=0.7\linewidth]{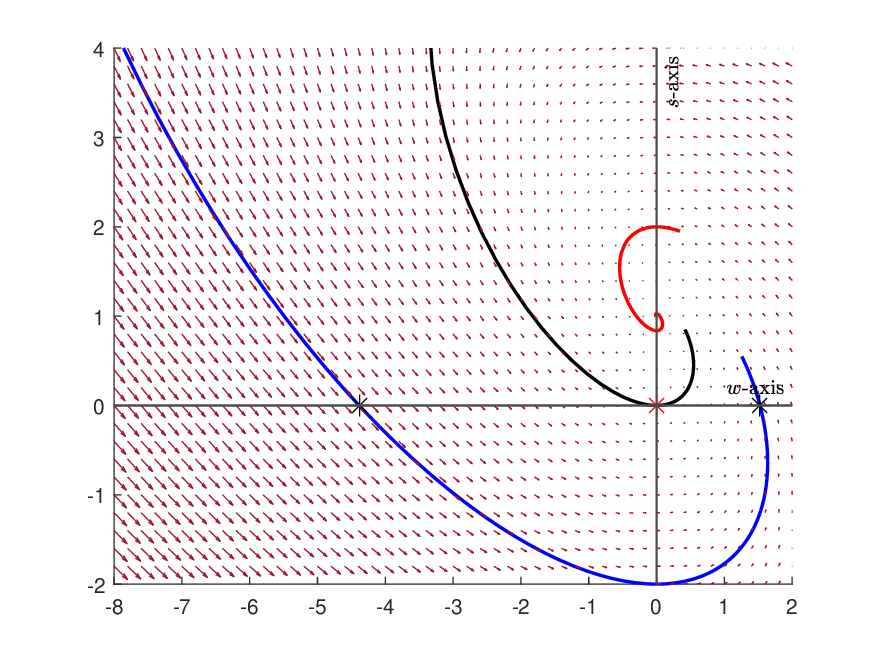}
        \caption{Illustration of the vector field.}
        \label{fig:vector_field}
\end{figure}

\end{remark}
Next, we prove a lemma that verifies the claim made in the construction of the critical threshold.

\begin{lemma}\label{lem: lifespan}
    Let $\beta\geq 2\sqrt{\kappa}$. Then the solution trajectories to the system \eqref{dynamics: w_s}, traced backwards in negative from the initial data
    \[
    w(0) = 0, \ \ \ s(0) = 0,
    \]
    remain confined within the region $R_1 = \{(w,s) : w<0, s>0\}$ for all time $t<0$. In other words, $C_s \subset R_1$ and $C_b \subset R_1$.
\end{lemma}
\begin{proof}
    Consider the case $\beta> 2\sqrt{\kappa}$, where solution trajectory is given by \eqref{eq:strong}. We first observe the asymptotic behavior 
    \[
    \lim_{t\to -\infty} s(t) = \infty, \ \ \ \lim_{t\to -\infty} w(t) = -\infty.
    \]
    In addition, we note that $s(t)$ and $w(t)$ both have at most one local extremum. Since $s'(0) = w(0) = 0$, this extremum for $s(t)$ is achieved, implying $s(t)$ is monotone increasing for all $t< 0$. For $w(t)$, we compute its critical point
    \[
    w'(t) = \frac{\lambda_1 \lambda_2}{\lambda_2-\lambda_1} \bigg(-\lambda_1 e^{-\lambda_1 t}+\lambda_2 e^{-\lambda_2 t}\bigg).
    \]
    Equating $w'(t)$ to $0$ and solve for $t$ to obtain
    \[
    t = \frac{1}{\lambda_2-\lambda_1}\ln\bigg(\frac{\lambda_2}{\lambda_1}\bigg) >0.
    \]
    Hence we see that $w(t)$ is monotone decreasing for all $t<0$. Hence $C_s \subset R_1$. Similarly, we conclude $C_c \subset R_1$.
\end{proof}

The quantity $s^*$ which defines the point of intersection earlier and the domains of the Lyapunov functions \eqref{lya:P} and \eqref{lya:N} introduced below is determined from explicitly calculating the $s$-axis intercept of the solution trajectory for the weakly damped system \eqref{dynamics: w_s}. We solve for the first negative time $t^*$ at which $w(t^*) = 0$,
\begin{align*}
    &e^{-\alpha t^*}\frac{2\kappa}{\sqrt{4\kappa-\beta^2}}\sin(\omega t^*) = 0 \ \Longleftrightarrow   \ \sin(\omega t^*) = 0.
\end{align*}
It's then straightforward to verify that $t^* = -\frac{\pi}{\omega}$. Hence we obtain
\begin{align*}
    s(t^*) &= \bigg(-\cos(-\pi)-\frac{\beta}{\sqrt{4\kappa-\beta^2}}\sin(-\pi)\bigg)e^{\frac{\alpha \pi }{\omega}} +1 =e^{\frac{\beta \pi}{\sqrt{4\kappa-\beta^2}}}+1. \numberthis \label{eq: s*}
\end{align*}

For analytical purpose, it is desirable to represent the critical threshold curve $C$ as a level set of suitable scalar functions of $(w,s)$. As in the undamped case $\beta=0$, the trajectories coincide with the level curves of the function \eqref{tt:lyapunov}. To express the solution trajectories and the threshold regions more succinctly, we construct the following Lyapunov functions, based on the approach utilized in \cite{tadmor2022critical}. This idea was first adopted in \cite{bhatnagar2020critical2} for the damped Euler-Poisson equations, and later extended to \cite{choi2025critical} and \cite{luan2026euler}. Let us introduce the first such Lyapunov function, 
\begin{equation}\label{lya:P}
     \L_P(w,s) = w+\sqrt{2P(s)},
\end{equation}
where the function $P(s)$ satisfies the following ODE
\begin{equation}\label{dynamics:P}
    \frac{dP}{ds} = \beta\sqrt{2P(s)}+\kappa(1-s),\quad  P(0) = 0.
\end{equation}
The next theorem precisely characterizes the relationship between the level sets of the above Lyapunov function and the solution trajectories to the system \eqref{dynamics: w_s}.

\begin{theorem}\label{thm:levelset}
    Suppose $\L_P(w,s)$ and $P(s)$ satisfy \eqref{lya:P} and \eqref{dynamics:P}, respectively, and 
    \[
    \textnormal{Range}\{s(t):t\in I \} \subset \textnormal{Dom}(P).
    \]
    Let $(w(t),s(t))$ be a solution to \eqref{dynamics: w_s} for all $t\in I$. If $\L_P(w(t_0), s(t_0)) = 0$ for some $t_0 \in I$, then the following holds:
    \[
    \L_P(w(t),s(t)) = 0, \quad \forall  \ t\in I.
    \]
\end{theorem}
\begin{proof}
    We differentiate $\L_P(w,s)$ with respect to time and compute
    \begin{align*}
        \frac{d}{dt}\L_P(w(t),s(t)) &= w^\prime(t) + \frac{1}{\sqrt{2P(s(t))}}\cdot\frac{dP}{ds}\cdot s^\prime(t)\\
        &= \kappa(1-s(t))-\beta w(t)+\frac{1}{\sqrt{2P(s(t))}}\;\bigg(\beta\sqrt{2P(s(t))}+\kappa(1-s(t))\bigg)\cdot w(t)\\
        &=\kappa(1-s(t))+\frac{\kappa(1-s(t))}{\sqrt{2P(s(t))}}\,w(t)\\
        &=\frac{\kappa(1-s(t))}{\sqrt{2P(s(t))}} \, \L_P(w(t),s(t)).
    \end{align*}
Consequently, we obtain
\begin{equation}\label{eq:lya:gronwall}
    \L_P(w(t),s(t)) = \L_P(w(t_0),s(t_0))\exp \bigg(\int_{t_0}^{t}\frac{\kappa(1-s(\tau))}{\sqrt{2P(s(\tau))}}\,\d \tau\bigg) = 0.
\end{equation}
\end{proof}
It follows from the Theorem ~\ref{thm:levelset} that the condition $P(0) = 0$ is specified so that the zero level set of $\L_P(w,s)$ represent the solution trajectory to \eqref{dynamics: w_s} passing through $(0,0)$ in the region  $R_1 := \{(w,s)\ : \ w<0, s > 0\}$. In the case of weak damping, the solution trajectory $C_w$ will exit the region $R_1$ at $(0,s^*)$. For expression of $s^*$, see \eqref{eq: s*}. To represent the part of $C_w$ in the region $R_2 : = \{(w,s) \ : \ w>0, s>0\}$, we construct an additional Lyapunov function, which satisfies the same property as stated in Theorem \ref{thm:levelset},
\begin{equation}\label{lya:N}
            \L_N(w,s) =w-\sqrt{2N(s)},
\end{equation}
where the function $N(s)$ satisfies the ODE
\begin{equation}\label{dynamics: N}
    \frac{dN}{ds} = -\beta\sqrt{2N(s)}+\kappa(1-s), \quad N(s^*) = 0, \quad s^* = e^{\frac{\beta \pi}{\sqrt{4\kappa-\beta^2}}}+1.
\end{equation}
A similar relation to \eqref{eq:lya:gronwall} holds for $\L_N$:
\begin{equation}\label{eq: P}
     \L_N(w(t),s(t)) = \L_N(w(t_0),s(t_0))\exp \bigg(\int_{t_0}^{t}\frac{\kappa(s(\tau)-1)}{\sqrt{2N(s(\tau))}}\,\d \tau\bigg).
\end{equation}
The choice of $s^*$ ensures that the union of sets $\{\L_p(w,s)=0\}$  and $\{\L_N(w,s)=0\}$ constitutes the single solution trajectory defined in ~\eqref{eq:weak:trajectory}.

Next, we present a version of the comparison principle based on the Lyapunov functions \eqref{lya:P} and \eqref{lya:N}, as a corollary from \ref{thm:levelset}. 

\begin{proposition}[Comparison principles]\label{comparison_principle}
Let $(w(t),s(t))$ be a solution to the system \eqref{dynamics: w_s}. Denote by $C$ the solution trajectory over the time interval $(t_0,t_1)$, namely 
\[
C: = \{(w(t),s(t)) : \ t\in (t_0,t_1)\}.
\]
Then the following results hold:
\begin{itemize}
    \item Let $R_1 = \{(w,s)\ : \ w<0, s > 0\}$ and suppose $C\subset R_1$, then
    \begin{equation}
        \L_P(w(t_0),s(t_0)) \leq 0  \ \ \Longleftrightarrow \ \  \L_P(w(t),s(t)) \leq 0, \ \ \forall \ \ t\in [t_0,t_1] .
    \end{equation}
    In addition,
    \begin{equation}
         \L_P(w(t_0),s(t_0)) > 0  \ \ \Longleftrightarrow \ \  \L_P(w(t_1),s(t_1)) > 0, \ \ \forall \ \ t\in [t_0,t_1] .
    \end{equation}
    \item Let $R_2 = \{(w,s)\ : \ w > 0, s>0\}$ and suppose $C \subset R_2$, then
     \begin{equation}
        \L_N(w(t_0),s(t_0)) < 0  \ \ \Longleftrightarrow \ \  \L_N(w(t_1),s(t_1)) < 0, \ \ \forall \ \ t\in [t_0,t_1] . 
    \end{equation}
    In addition,
    \begin{equation}
         \L_N(w(t_0),s(t_0)) \geq 0  \ \ \Longleftrightarrow \ \  \L_N(w(t_1),s(t_1)) \geq 0, \ \ \forall \ \ t\in [t_0,t_1] . 
    \end{equation}
\end{itemize}
\end{proposition}
\begin{proof}
It suffices to establish the strict inequalities, as the non-strict cases follow directly from the integral representation \eqref{eq:lya:gronwall}. While we could  control the integral in the exponential term in \eqref{eq:lya:gronwall} by following the approach in \cite[Proposition 4.7]{choi2025critical}, the result is an immediate consequence of the uniqueness theorem for autonomous ODE system.

Let $\mathcal{C}_0$ denote the zero level set of $\L_P(w(t),s(t))$, which represents the unique solution trajectory of \eqref{dynamics: w_s} passing through the origin. Consider an arbitrary trajectory $\mathcal{C}$ starting at $(w(t_0), s(t_0))$ such that $\L_P(w(t_0), s(t_0)) \neq 0$. Since $\mathcal{C}$ and $\mathcal{C}_0$ are distinct integral curves of an autonomous system, they must be disjoint for all time. Consequently, $\L_P(w(t), s(t))$ can never vanish. The analogous strict inequality for $\L_N$ follows by the same argument.
\end{proof}

Next, we utilize the above comparison principles to give precise construction and characterization of the super- and subcritical regions.

\begin{itemize}
    \item [(I)] In the region $R_1 = \{(w,s)\ : \ w<0, s > 0\}$, let $C_1$ be a trajectory of \eqref{dynamics: w_s} that passes through $ (0,0)$, which can be represented as the level set $\L_p(w,s) = 0$ by Theorem \ref{thm:levelset}. We claim the lifespan of $P(s)$ is given as follows, 
    \begin{equation}
        \textnormal{Dom}(P) = \begin{cases}
            [0, \infty) \ \ \  \beta\geq 2\sqrt{\kappa},\\
            [0, s^*] \ \ \ 0 < \beta < 2\sqrt{\kappa}.
        \end{cases}
    \end{equation} 
    The claim follows directly from Lemma \ref{lem: lifespan}. Depending on the damping regime, there are two situations:
    \medskip
    \begin{enumerate}
        \item If $\beta \geq 2\sqrt{\kappa}$, $C_1$ won't exit the region $R_1$ and can be represented as $ \{(w,s): w = -\sqrt{2P(s)}, \ s\in [0, \infty)\}.$ In this case, the subcritical region is characterized as 
          \[
    \Sigma_\flat = \Big\{(w,s): w > -\sqrt{2P(s)}, \ s>0\Big\}.
    \]
     Namely, the region $\Sigma_\flat$ consists of the open set in $\R\times \R_+$ to the right of the curve $C_1$. The supercritical region $\Sigma_\sharp$ is defined to be $(\R\times \R_+ )\setminus \Sigma_\flat$, namely: 
     \[
      \Sigma_\sharp= \Big\{(w,s): w \leq -\sqrt{2P(s)}, \ s>0\Big\}.
     \]
     The initial condition of $P$ is chosen so that $\L_P(0,0) = \sqrt{2P(0)} = 0$.
        
        \medskip
        \item If $ \beta < 2\sqrt{\kappa}$, $C_1$ exits the region at $(0, s^*)$, and can be represented as $ \{(w,s): w =-\sqrt{2P(s)}, \ s\in [0, s^*]\}.$ We continue our construction. 
    \end{enumerate}
    
    \medskip
    \item [(II)] In the region $R_2 = \{(w,s)\ : \ w > 0, s>0\}$, let $C_2$ be a trajectory of \eqref{dynamics: w_s} that passes through $(0, s^*)$. The lifespan of $N(s)$ is 
       \begin{equation}
        \textnormal{Dom}(N) =  [0, s^*].
    \end{equation} The trajectory takes the form $ C_2 = \{(w,s): w = \sqrt{2N(s)}, \ s\in [0, s^*]\}$, and it exits the region at $(w^*, 0)$. In this case, the subcritical region is given as
          \[
    \Sigma_\flat = \Big\{(w,s): \sqrt{2N(s)}>w>  -\sqrt{2P(s)}, \ s\in (0, s^*]\Big\}.
    \]
   Namely, the region $\Sigma_\flat$ consists of the open set in $\R\times \R_+$ bounded by the curve $C_1 \cup C_2$. The supercritical region $\Sigma_\sharp$ is defined to be $(\R\times \R_+ )\setminus \Sigma_\flat$. The initial condition of $N$ is chosen so that $\L_N(0, s^*) = -\sqrt{2N(s^*)} = 0.$
\end{itemize}
\begin{remark}

To verify that the constructed set $\Sigma_\flat$ is \emph{invariant} in the sense of \ref{region}, we must show that any solution trajectory initialized within $\Sigma_\flat$ never escapes. We focus on the weak damping case ($\beta < 2\sqrt{\kappa}$), as it includes the case when $\beta\geq 2\sqrt{\kappa}$.

The strict comparison principles established in Proposition \ref{comparison_principle} ensure that a trajectory cannot cross the lateral boundaries defined by the level sets $\L_P(w,s)=0$ and $\L_N(w,s)=0$. Thus, the only potential exit is through the bottom boundary segment $\{(w,0) : 0 < w \leq w^*\}$. However, on this segment, the vertical velocity of the flow is given by $s'(t) = w(t) > 0$. This implies that the vector field points strictly \emph{into} the region $s>0$. 
\end{remark}
The comparison principle based approach yields another \emph{sharp} characterization of the critical thresholds.
\begin{theorem}[Global regularity]\label{thm:GWP_lya}
    Let $(\rho, \u, \phi)$ be a classical solution of the damped EMA system \eqref{eqs:EMA} with radial symmetry \eqref{eq:radial}. Suppose that the initial data $(\u_0, \phi_0)$ lie within the following subcritical region:
    \begin{itemize}
        \item Strong or critical damping $(\beta \geq 2\sqrt{\kappa})$: 
\begin{equation}\label{sub:sm:uphi}
      u_0'(r) > \big(\phi_0''(r)-1\big)\sqrt{2P\bigg(\frac{1}{1-\phi_0''(r)}\bigg)}.
\end{equation}
        \item Weak damping $(\beta < 2\sqrt\kappa)$:
        \begin{equation}\label{sub:weak:uphi}
             \big(1-\phi_0''(r)\big)\sqrt{2N\bigg(\frac{1}{1-\phi_0''(r)}\bigg)} > u_0'(r) > \big(\phi_0''(r)-1\big)\sqrt{2P\bigg(\frac{1}{1-\phi_0''(r)}\bigg)},
        \end{equation}
    \end{itemize}
  for all $r>0$, then $\rho$ and $\grad\u$ are uniformly bounded in all time. Otherwise, suppose there exists an $r>0$ such that either \eqref{sub:sm:uphi} or \eqref{sub:weak:uphi} is violated in their respective damping regimes. Then there exist a location $r_c$ and a finite time $T_c$ at which
    \begin{equation}
        \lim_{t \ \rightarrow \ T_C^-} \pa_ru(r_c,t) = -\infty, \quad \lim_{t \ \rightarrow \ T_C^-}\rho(r_c,t) = \infty.
    \end{equation}
\end{theorem}

\begin{proof}
The global regularity is a direct consequence of the invariance of the subcritical region $\Sigma_\flat$. As established previously, any trajectory initialized within $\Sigma_\flat$ remains confined to this set for all $t>0$. Consequently, by invoking the same arguments utilized in the proof of Theorem \ref{thm: GWP}, we obtain global regularity.

Conversely, for data in the supercritical region, the finite-time blowup scenario follows from a phase plane analysis argument. One can follow the approaches highlighted in either \cite[Proposition 4.1]{choi2025critical} or \cite[Proposition 4.5]{luan2026euler} almost verbatim. This completes the proof.
\end{proof}

\begin{remark}
    Both threshold constructions in Section 4 and 5 are equivalent. Indeed, the equality case in Theorem ~\ref{thm: GWP} is precisely the solution trajectory of \eqref{dynamics: w_s} passing through $(w,s)=(0,0)$. On the other hand, Theorem ~\ref{thm:levelset} shows that the zero level sets $\CL_p =0$ and, in the weakly damped case, $\CL_N=0$, are solution trajectories to \eqref{dynamics: w_s} passing through $(0,0)$. By uniqueness for the autonomous ODE system \eqref{dynamics: w_s}, these curves coincide. Thus Theorem ~\ref{thm:GWP_lya} is an implicit Lyapunov-function characterization of the explicit thresholds in Theorem ~\ref{thm: GWP}.
\end{remark}

\begin{remark}
Let us recover the critical threshold for the undamped EMA system with radial symmetry using Theorem \ref{thm:GWP_lya}. Recall that this condition is given as 
\[
|u_0'(r)| < \sqrt{\kappa(1-2\phi_0''(r))}.
\]

In the absence of damping ($\beta=0$), The Lyapunov function takes the form
    $\L(w,s) = w^2+\kappa(1-s)^2$. As verified earlier, $\frac{\d}{\d t} \L(w,s) = 0$, implying that $w^2+\kappa(1-s)^2 = c$ for some constant $c$. Since the critical threshold curve crosses the origin $(0,0)$ in the $(w,s)$-plane, we must have $c = \kappa$. This gives the explicit representation for both $P(s)$ and $N(s)$:
    \[
    P(s) = N(s) = \frac{\kappa-\kappa(1-s)^2}{2}.
    \]
   Substituting $s = \frac{1}{1-\phi_0''(r)}$ into both functions yields
   \[
   P\bigg(\frac{1}{1-\phi_0''(r)}\bigg) = N\bigg(\frac{1}{1-\phi_0''(r)}\bigg) = \frac{\kappa \big(1-2\phi_0''(r)\big)}{2(1-\phi_0''(r)\big)^2}.
   \]
    Then using the subcritical condition \eqref{sub:weak:uphi} in Theorem \ref{thm:GWP_lya}, we obtain the critical threshold condition:
    \[
    \sqrt{\kappa \big(1-2\phi_0''(r)\big)} > u_0'(r)> -\sqrt{\kappa \big(1-2\phi_0''(r)\big)}.
    \]
    This recovers the critical threshold derived in \cite{tadmor2022critical}.
\end{remark}

\section{Conflict of Interest}
The author declares that they have no known competing financial interests or personal relationships that could have appeared to influence the work reported in this paper.

\section{Data Availability}
Data sharing is not applicable to this article as no datasets were generated or analyzed during the current study.

\appendix

\section{Local well-posedness and regularity criterion}\label{sec:LWP}
Let $\O(\x,t)$ be a vector-valued radial function defined as 
\begin{equation}\label{def:Omega}
    \O(\x,t) = \begin{bmatrix}
        p(|\x|,t) \\
        q(|\x|,t) \\
        \mu(|\x|,t)\\
        \nu(|\x|,t)
    \end{bmatrix} \\ =\begin{bmatrix}
        \pa_r u(r,t) \\
        \frac{u(r,t)}{r}\\
        \pa_r^2\phi(r,t)\\
        \frac{\pa_r\phi(r,t)}{r}
    \end{bmatrix}.
\end{equation}

This allows us to compress the dynamics \eqref{eq:qnu} and \eqref{eq:pmu} into a vector-valued equation of $\O$:
\begin{equation}\label{eq: Omega}
    \pa_t \O + (\u\cdot \grad)\O = \F(\O), \quad \ \F(\O) = \begin{bmatrix}
        -\Omega_1^2 - \kappa \Omega_3 - \beta \Omega_1\\
        -\Omega_2^2 - \kappa \Omega_4 - \beta \Omega_2\\
        \Omega_1(1-\Omega_3)\\
        \Omega_2(1-\Omega_4)
    \end{bmatrix}.
\end{equation}
Equivalently, we write 
\begin{equation}
    \pa_t \Omega_i + \div(\Omega_i\u) = \Tilde{F_i}(\O), \quad i \in \{1,2,3,4\}
\end{equation}
where the forcing term $\Tilde{F_i}$ is defined as 
\begin{equation}\label{eq:Ftilde}
    \Tilde{\F}(\O):= \F(\O) + (\div \u)\O = \begin{bmatrix}
          -\Omega_1^2 - \kappa \Omega_3 - \beta \Omega_1+\Omega_1(\Omega_1+(n-1)\Omega_2)\\
        -\Omega_2^2 - \kappa \Omega_4 - \beta \Omega_2+\Omega_2(\Omega_1+(n-1)\Omega_2)\\
        \Omega_1(1-\Omega_3)+\Omega_3(\Omega_1+(n-1)\Omega_2)\\
        \Omega_2(1-\Omega_4)+\Omega_4(\Omega_1+(n-1)\Omega_2)
    \end{bmatrix}.
\end{equation}
Here we used the relation 
\begin{equation}\label{eq:div_u}
    \div \u = p(r,t)+ (n-1)q(r,t) = \Omega_1 + (n-1)\Omega_2.
\end{equation}
We prove Theorem ~\ref{thm: LWP}.
\begin{proof}
    Given any $s\geq 0$, we denote the fractional differential operator by $\Lambda^s: = (-\Delta)^{\frac{s}{2}}$ and the homogeneous semi-norm by 
    \[
    \|f\|_{\dot{H}^s(\R^n)} = \|\Lambda^s f\|_{L^2(\R^n)}.
    \] 
   The notation $\lesssim$ will be used throughout the proof, where $A \lesssim B$ means there exists a constant $C$, depending on the parameters and initial data, such that $A\leq CB$. In addition, we use $\pa_j$ to denote $\frac{\pa}{\pa x_j}$. Let the energy $E_s(t)$ be defined as 
    \[
    E_s(t) := \frac{1}{2}\|\O\|_{H^s(\R^n)}^2 = \frac{1}{2}\|\O\|_{L^2(\R^n)}^2 + \frac{1}{2}\|\Lambda^s \O\|^2_{L^2(\R^n)}.
    \]
     Next, we provide an a priori energy estimate. Let us first consider the $L^2$ energy estimate:
     \begin{align*}
      \frac{1}{2}\frac{d}{dt} \|\O\|^2_{L^2} &=\sum_{i=1}^4\int_{\R^n} (\Tilde{F}_i(\O)-\div(\Omega_i \u))\cdot \Omega_i \ \d x\\
      &= \sum_{i=1}^4 \int_{\R^n} -\frac{1}{2}\Omega_i^2 (\div \u) +\Tilde{F}_i(\O)\cdot \Omega_i \  \d x\\
      &\leq \sum_{i=1}^4 (\|\div \u\|_{L^\infty} \|\Omega_i\|_{L^2}^2 + \|\Tilde{F}_i(\O)\|_{L^2} \|\Omega_i\|_{L^2} )\\
      &\leq \sum_{i=1}^4 (\|\div \u\|_{L^\infty} \|\Omega_i\|_{L^2}^2 + (1+\|\O\|_{L^\infty})  \|\Omega_i\|^2_{L^2} )\\
      &\lesssim (1+ \|\div \u\|_{L^\infty}  +\|\O\|_{L^\infty}  )\|\O\|^2_{L^2}  \lesssim (1+\|\O\|_{L^\infty}) \|\O\|_{L^2}^2.
     \end{align*}
    Observe the quadratic dependence of $\Tilde{F}_i$ on $\O$ leads to the following inequality:
    \[
    \|\Tilde{\F}(\O)\|_{L^2} \lesssim (1+\|\O\|_{L^\infty})\|\O\|_{L^2},
    \]
    which is used in the third to last line. Next, we estimate the $\dot{H}^s$ energy:
    \begin{align*}
        \frac{1}{2}\frac{d}{dt} \|\Lambda^s \O\|_{L^2}^2 &= 
        \sum_{i=1}^4\int_{\R^n} \bigg[\Lambda^s\bigg(\Tilde{F}_i(\O)-\div(\Omega_i \u)\bigg)\bigg]\Lambda^s\Omega_i \ \d x \\
        &= -\sum_{i=1}^4\sum_{j=1}^n \int_{\R^n} \bigg(\Lambda^s \pa_{j}(\Omega_i\cdot u_j)\cdot\Lambda^s \Omega_i \bigg) \ \d x +\sum_{i=1}^4 \int_{\R^n} \Lambda^s \Tilde{F}_i(\O) \cdot \Lambda^s \Omega_i \ \d x \\
        &:= I + II.
    \end{align*}
    Let us estimate $I$ first. Here we use $[\cdot,\cdot]$ to denote the commutator, whose operation is given as
\[[\mathcal{J},f]g = \mathcal{J}(fg)-f\mathcal{J}g.\]
It follows that $\Lambda^s\pa_j(u_j \Omega_i) = [\Lambda^s \pa_j, u_j]\Omega_i+ u_j\Lambda^s \pa_j \Omega_i$. This gives
    \begin{align*}
        I &= \frac{1}{2}\sum_{i=1}^4\sum_{j=1}^n \int_{\R^n}  (\Lambda^s \Omega_i)^2 \cdot \pa_{j}u_j \ \d x + \sum_{i=1}^4\sum_{j=1}^n \int_{\R^n} -[\Lambda^s\pa_{j}, u_j] \ \Omega_i \cdot \Lambda^s \Omega_i \ \d x  := I_1 + I_2.
    \end{align*}
For $I_1$, we observe from \eqref{eq:div_u} the following holds
\begin{align*}
     \|\div \u\|_{L^\infty}   &\leq \|\Omega_1\|_{L^\infty}+ (n-1)\|\Omega_2\|_{L^\infty} \lesssim \|\O\|_{L^\infty}.
\end{align*}
It then follows
\begin{align*}
    I_1  &= \frac{1}{2}\sum_{i=1}^4 \int_{\R^n} (\Lambda^s \Omega_i)^2 \ \div\u \ \d x \lesssim \|\div \u\|_{L^\infty} \|\Lambda^s \Omega\|_{L^2}^2 \lesssim \|\O\|_{L^\infty} \|\Lambda^s \Omega\|_{L^2}^2. 
\end{align*}
To estimate $I_2$, we use a Kato-Ponce type commutator estimate (see \cite{kato1988commutator}),
\begin{align*}
    I_2 &\leq\sum_{i=1}^4\sum_{j=1}^n\|[\Lambda^s\pa_{j}, u_j]\ \Omega_i\|_{L^2} \|\Lambda^s \Omega_i\|_{L^2}\\
    &\lesssim \sum_{i=1}^4\sum_{j=1}^n ( \|\pa_j u_j\|_{L^\infty}\|\Lambda^s \Omega_i\|_{L^2}+ \|\Lambda^s \pa_j u_j\|_{L^2}\|\Omega_i\|_{L^\infty} ) \|\Lambda^s \Omega_i\|_{L^2}\\
    &\lesssim \|\div \u\|_{L^\infty} \|\Lambda^s \O\|_{L^2}^2 + \|\Lambda^s (\div \u)\|_{L^2}\|\O\|_{L^\infty}\|\Lambda^s \O\|_{L^2} \lesssim\|\O\|_{L^\infty}  \|\Lambda^s \O\|^2_{L^2},
\end{align*}
where in the last line we use $\|\div \u\|_{L^\infty} \lesssim \|\O\|_{L^\infty}$ and the relation
\begin{align*}
    \|\Lambda^s (\div \u)\|_{L^2} &\leq \|\Lambda^s \Omega_1\|_{L^2} + (n-1) \|\Lambda^s \Omega_2\|_{L^2} \lesssim \|\Lambda^s \O\|_{L^2}. 
\end{align*}
Together, the estimates of $I_1$ and $I_2$ give
\begin{equation}\label{estimate:I}
    I\lesssim \|\O\|_{L^\infty} \|\Lambda^s \O\|_{L^2}^2.
\end{equation}
Next we estimate $II$. First observe the quadratic dependence of $\Tilde{\F}$ on $\O$ in \eqref{eq:Ftilde} implies 
\[
\|\Lambda^s (\Tilde{\F})(\O)\|_{L^2} \lesssim (1+ \|\O\|_{L^\infty}) \|\Lambda^s \O \|_{L^2}.
\]
Hence, we obtain
\begin{align*}
    II &=  \sum_{i=1}^4 \int_{\R^n} \Lambda^s \Tilde{F}_i(\O) \cdot \Lambda^s \Omega_i \ \d x \\
    &\leq \sum_{i=1}^4  \|\Lambda^s \Tilde{F}_i(\O)\|_{L^2} \|\Lambda^s \Omega_i\|_{L^2}\lesssim (1+\|\O\|_{L^\infty})\|\Lambda^s \O\|_{L^2}^2 \numberthis \label{estimate:II}.
\end{align*}
Combining \eqref{estimate:I} and \eqref{estimate:II} to get
\[
\frac{1}{2}\dfrac{d}{dt} \|\Lambda^s \O\|_{L^2}^2 \lesssim (1+\|\O(\cdot, t)\|_{L^\infty})\|\Lambda^s \O\|_{L^2}^2.
\]
Consequently, we have the following estimate on the energy $E_s(t)$:
\begin{equation}\label{eq:Es}
    \frac{d}{dt}E_s(t) \lesssim (1+ \|\O(\cdot, t))\|_{L^\infty})E_s(t).
\end{equation}

If $s>\frac{n}{2}$, then Sobolev embedding gives local well-posedness and we apply Gr\"onwall inequality to \eqref{eq:Es} to obtain
\[
E_s(t) \leq E_s(0) \exp\bigg[C \int_0^T  (1+\|\O(\cdot, t)\|_{L^\infty})\ \d t\bigg] .
\]
Hence, $E_s(t)$ is bounded as long as \eqref{eq:BKM} holds.
\end{proof}

\bibliographystyle{plain}
\bibliography{bib_EMA_revision}

\end{document}